\documentclass[a4paper, 11pt]{amsart}

\usepackage{euscript}
\usepackage{amsmath,amsthm}
\usepackage{amssymb}
\usepackage{mathtools}
\usepackage{dsfont}
\usepackage{geometry}
\usepackage{xcolor}
\usepackage{tikz-cd}
\usepackage[colorlinks=true,linkcolor=red,citecolor=red]{hyperref}
\usepackage{enumerate}
\usepackage[shortlabels]{enumitem}
\usepackage{enumitem}
\usepackage{hyperref}

\usepackage[matrix,arrow,curve,frame]{xy}

\title{Rigidification of cubical quasi-categories}
\author{Pierre-Louis Curien}
\address[P.-L. Curien]{
CNRS, Univ. Paris Cit\'e, IRIF, INRIA}
\email{curien@irif.fr}
\author{Muriel Livernet}
\address[M. Livernet]{
Univ Paris Cit\'e, Institut de Math\'ematiques de Jussieu-Paris Rive Gauche, CNRS, Sorbonne Universit\'e,  Paris, France}
\email{muriel.livernet@imj-prg.fr}
\author{Gabriel Saadia}
\address[G. Saadia]{
Dept. of Mathematics, Stockholm University}
\email{gabriel.saadia@math.su.se}

\newtheorem{thm}[subsubsection]{Theorem}
\newtheorem{lem}[subsubsection]{Lemma}
\newtheorem{prop}[subsubsection]{Proposition}
\newtheorem{cor}[subsubsection]{Corollary}

\theoremstyle{definition}
\newtheorem{ex}[subsubsection]{Example}

\newtheorem{notation}[subsubsection]{Notation}

\newtheorem{defn}[subsubsection]{Definition}
\newtheorem{rmk}[subsubsection]{Remark}

\newcommand{\ie}{\emph{i.e.} }
\newcommand{\C}{\mathfrak{C}}
\newcommand{\Cbox}{\mathfrak{C}^{\square}}
\newcommand{\Nbox}{N^{\square}}
\newcommand{\cSdp}{c\EuScript{S}et_{*,*}}
\newcommand{\sCdp}{s\EuScript{C}at_{*,*}}
\newcommand{\Set}{\EuScript{S}et}
\newcommand{\Cat}{\EuScript{C}at}
\newcommand{\Nec}{\EuScript{N}ec}
\newcommand{\slicecat}[2]{{#1}\downarrow{#2}}
\newcommand{\enrichedcat}[1]{ \it{#1}\mbox{-}\Cat}

\DeclareMathOperator*{\hocolim}{hocolim}
\DeclareMathOperator*{\colim}{colim}
\DeclareMathOperator*{\Hom}{Hom}

\DeclareMathOperator*{\Ob}{Ob}
\DeclareMathOperator*{\id}{id}
\DeclareMathOperator*{\SubNeck}{SubNeck}

\makeatletter
\@namedef{subjclassname@2020}{\textup{2020} Mathematics Subject Classification}
\makeatother

\date{\today}



\begin{document}
\maketitle

\begin{abstract}
We construct a cubical analogue of the rigidification functor from quasi-categories to simplicial categories present in the work of  Joyal and Lurie. We define a functor  $\Cbox$ from the category $c\Set$ of 
cubical sets of 
Doherty-Kapulkin-Lindsey-Sattler
 to the category $s\Cat$ of (small) simplicial categories.  We show that this rigidification functor establishes a Quillen equivalence between the Joyal model structure on 
 $c\Set$ (as it is called by the four authors) and Bergner's model structure on  $s\Cat$. We follow  the approach to  rigidification of Dugger and Spivak, adapting 
 their framework of necklaces to the cubical setting.
\end{abstract}

\section*{Introduction}

The last decades have seen an explosion of the use of $\infty$-categories  in various fields such as algebraic topology, algebraic geometry, 
homotopy type theory. During the first years of 2000, various definitions of $\infty$-categories have emerged, starting from the notion of 
quasi-categories developed by Joyal in \cite{J} and by Lurie in \cite{L} based on the definition of Boardman-Vogt in \cite{BV73}.  Other definitions have been 
explored such as enriched categories in spaces (or Kan complexes), or complete Segal spaces to name a few. By model of  $\infty$-categories 
we mean  a category with a Quillen model category structure whose fibrant-cofibrant objects are the  $\infty$-categories in consideration and whose notion of weak 
equivalence corresponds to a good notion of equivalence of $\infty$-categories. Bergner's 
book \cite{B18} clearly explains these different models and the Quillen equivalences relating them.\\

For instance, the model for quasi-categories is the Joyal model category structure on the category $s\Set$ of simplicial sets, while the model for 
categories enriched in Kan complexes is the Bergner model structure on the category $s\Cat$ of (small) simplicial categories. There exists 
a so-called rigidification functor $\C^{\Delta}$ from $s\Set$ to the category $s\Cat$ which is a Quillen equivalence  between these two models.
This functor is called rigidification because simplicial categories have a  strict composition of 1-morphisms, as opposed to quasi-categories where 
only weak compositions exist.  The construction of the rigidification as well as the proof that it yields a  Quillen equivalence have been achieved 
first  in an unpublished manuscript of Joyal \cite{J2},  then by Lurie \cite{L}, and then, by  Dugger and Spivak in \cite{DS11}.
Dugger and Spivak build their rigidification functor using a technical tool, necklaces, and prove that there is a zig-zag of weak equivalences of 
simplicial categories between their construction and Lurie's one. The key idea of this construction is the following: given an ordered 
simplicial set  $X$, the 
simplicial set  $\C^\Delta(X)(a,b)$ is the nerve of a poset whose objects are directed paths and relations are generated by $2$-simplices in $X$. In 
case $X=\Delta^n$, one obtains the subset lattice of an ordered set. \\

Cubical sets have been often considered as an alternative to simplicial sets in  combinatorial topology, including in the early 
work of Kan and Serre (see e.g. \cite{Ser51}). It has been also developed in  computer science, in particular in concurrency theory (see e.g. 
\cite{Pratt91}, \cite{Gau08} and \cite{FGHM16}) and in homotopy type theory (see e.g. \cite{CCHM18}).
In \cite{DKLS},
Doherty, Kapulkin, Lindsey, and Sattler have defined a notion of cubical quasi-category, and have 
constructed a model category structure on cubical sets,  analogous to the Joyal structure on simplicial sets, whose fibrant-cofibrant objects are cubical 
quasi-categories. They also show that
the categories  $c\Set$ of cubical sets and $s\Set$ are related  by two adjunctions. The first one  is $T\dashv U$, where $T:c\Set \rightarrow s\Set$ is a  
triangulation functor, and the second one is $Q\dashv \int$, where $Q:s\Set\rightarrow c\Set$ is a ``cubification'' functor implementing  simplices as 
cubes with some degenerate faces. Both  give rise to Quillen equivalences between these model category structures. So that cubical quasi-categories provide another 
definition for the notion of $\infty$-category. By plugging together the Quillen equivalences  $T:c\Set \rightarrow s\Set$  and $\C^{\Delta}:s\Set\rightarrow s\Cat$ of {\em triangulation} 
and rigidification,  we get a Quillen equivalence between Joyal model structure of $c\Set$ 
and Bergner's model structure on $s\Cat$. \\

The goal of this paper is to build a different, direct
Quillen equivalence $\C^\Box:c\Set\rightarrow s\Cat$ using directed paths in the spirit of Dugger and Spivak. Note that the same notion of directed path is used in 
directed homotopy theory with applications to computer science. We refer the interested reader to the papers by Ziemia\'nski (\cite{Zie17}, \cite{Zie20}),  and the references therin. In particular, for 
a representable cubical set $\square^n$ and for two vertices $a$ and $b$, the simplicial set $\C^\Box(\square^n)(a,b)$ is the nerve of a poset whose objects are 
directed paths from $a$ to $b$ in the $n$-cube and relations are generated by $2$-cubes in $\square^n$. We prove a result of independent interest, namely that this poset 
is isomorphic to the weak Bruhat order on a symmetric group. Following closely the techniques developed by Dugger and Spivak, we prove that the  functor $\Cbox$ is the
 left adjoint of a Quillen equivalence between two models of $\infty$-categories, making use  this time  of the {\em cubification} equivalence of \cite{DKLS}, by showing that 
 $\C^{\Delta}$ factorises through cubification via our rigidification, up to natural homotopy, \ie $\Cbox \circ Q \stackrel{\sim}\Rightarrow \C^{\Delta}$, and then concluding 
 by the 2 out of 3 property.
 
 \subsection*{Plan of the paper}
In Section 
\ref{recollection-section}, we recall Bergner's model structure and the material from \cite{DKLS} needed for our purposes.
Section \ref{necklace-section} is devoted to the study of paths and necklaces (adapted from \cite{DS11,DS2}). 
We define our rigidification functor and study its properties in Section \ref{rigidification-section}.
The Quillen equivalence is established in Section \ref{main-section}. The two appendices \ref{category-appendix} and \ref{combinatorics-appendix} 
deal with relevant categorical  and combinatorial matters, respectively.

\tableofcontents

\section{Recollection and Notation} \label{recollection-section}

\subsection{Simplicial rigidification}\label{S:simp_rig} 
We recall the {\sl Bergner model structure} of $s\Cat$. It is the enriched model structure coming from the usual Kan-Quillen model 
structure on $s\Set$.

\begin{itemize}
    \item Weak equivalences are {\sl Dwyer-Kan equivalences}, that is, functors $F : \EuScript{C}\to\EuScript{D}$ such that
        \begin{itemize}
            \item $\pi_0(F) : \pi_0(\EuScript{C})\to\pi_0(\EuScript{D})$ is essentially surjective, and
            \item for all $x,y\in\Ob(\EuScript{C})$, the map $F_{x,y} : \EuScript{C}(x,y)\to\EuScript{D}(Fx,Fy)$ is a Kan-Quillen equivalence.
        \end{itemize}
    \item Fibrations are {\sl Dwyer-Kan fibrations}, that is, functors $F : \EuScript{C}\to\EuScript{D}$ such that
        \begin{itemize}
            \item $\pi_0(F) : \pi_0(\EuScript{C})\to\pi_0(\EuScript{D})$ is an isofibration between categories, and
            \item for all $x,y\in\Ob(\EuScript{C})$, the map $F_{x,y} : \EuScript{C}(x,y)\to\EuScript{D}(Fx,Fy)$ is a Kan fibration.
        \end{itemize}
\end{itemize}

See \cite{B18} for more details.

Note that if $F$ happens to be  a bijection on objects, then $\pi_0(F)$ is a fortiori essentially surjective. This will be the case in our main result, so we will have to focus only on the second condition for DK equivalences.

For the next Proposition we use notation of Appendix \ref{S:wedge}.

\begin{prop}\label{adj_Sigma_Hom}
The following two functors form a Quillen adjunction
\[\xymatrix{
s\Set \ar@<1.3ex>[r]^-{\Sigma} \ar@<0.2ex>@{}[r]|-{\perp} & s\Cat_{*,*} \ar@<1ex>[l]^{\Hom},
}\]

where
\begin{itemize}
\item $s\Cat_{*,*}$ stands for the category of bipointed  (small) simplicial categories, with the model structure induced by the Bergner structure on $s\Cat$,
\item The model structure on $s\Set$ is the Joyal  structure,
\item $\Sigma(S)$ is the simplicial category with two objects $\alpha$ and $\omega$ and with only one non-trivial mapping space $\Hom(\alpha,\omega) = S$, and
\item $\Hom(\EuScript{C}_{x,y}) = \EuScript{C}(x,y)$.
\end{itemize}
\end{prop}
\begin{proof}
We check that $\Hom$ is right Quillen.
\end{proof}

The simplicial rigidification functor $\C^{\Delta} : s\Set \to s\Cat$ is obtained as a left Kan extension along the Yoneda functor. On the 
representables, $\C^{\Delta}$ is defined 
as follows: 
\begin{itemize}
\item $\Ob(\C^{\Delta}(\Delta^n)) = \{0,\ldots,n\}$.
\item For $i\leq j$, $\C^{\Delta}(\Delta^n)(i,j)$ is the nerve of the poset $\mathcal{P}(]i,j[)$ where $]i,j[$ is the set $\{i+1,i+2,\ldots,j-1\}$. The poset structure is given by subset inclusion. 
Note that this is the one point simplicial set if $j=i$ or $j=i+1$. For $i>j$,  $\C^{\Delta}(\Delta^n)(i,j)=\emptyset$.

\item Composition  $N(\mathcal{P}(]j,k[)) \times N(\mathcal{P}(]i,j[)) \rightarrow  N(\mathcal{P}(]i,k[))$ is induced by the function mapping 
$Y,X$ to $X\cup \{j\}\cup Y$.
\end{itemize}

The nerve functor $N:\Cat\rightarrow s\Set$ being monoidal, it induces a functor from categories enriched in categories to categories enriched in simplicial sets (see \cite[Chapter 3]{Rie14} for basics on enriched category theory).
We also call this functor the nerve functor and denote it by $N$. In particular, the simplicial category $\C^{\Delta}(\Delta^n)$ is obtained as the nerve of 
a poset-enriched category.

\begin{rmk}\label{R:objets} The simplicial rigidification functor is built by left Kan extension and so is cocontinuous, which implies in particular that the set of objects of $\C^{\Delta}(X)$ is in bijection with $X_0$. This is a general fact. Indeed, for a functor $F:I\rightarrow s\Cat$, the set of objects of the simplicial category $\colim F$ is in bijection with the colimit (in $\Set$) of the object functor $Ob\circ F:I\rightarrow \Set$, since 
the object functor is cocontinuous:  it is  left adjoint to the codiscrete functor $\Set \stackrel{coDisc}{\to} s\Cat$ sending 
a set $X$ to the simplicial category whose set of objects is $X$ and whose simplicial set of morphisms between any two objects is $\Delta^0$.
\end{rmk}

\begin{thm}[\cite{B18}] \label{sSet-sCat-quillen-equivalence}
The functor $\C^{\Delta} : s\Set \to s\Cat$ is the left adjoint of a Quillen equivalence between the Joyal model structure on $s\Set$ and the Bergner model structure on $s\Cat$.
\end{thm}

\subsection{Cubical quasi-categories}  \label{S:Q}

We  next present  the material of  \cite{DKLS} needed for our purposes.
There are different notions of cubical sets depending on whether one considers all or part of the negative and positive connections, the diagonals, and the symmetries.
We follow \cite{DKLS} in using  the negative connections only, and we denote this category by $\square$.

The category $\square$ is the subcategory of the category of posets whose objects  are $[1]^n, n\geq 0,$  and whose  morphisms are generated by
 \begin{itemize}
\item the faces $\partial^n_{i,\epsilon}: [1]^{n-1}\rightarrow [1]^n$ ( $1\leq i\leq n, 
\epsilon\in\{0,1\}$), consisting in inserting $\epsilon$ at the $i$-coordinate,
\item the degeneracies $\sigma_i^n: [1]^n\rightarrow [1]^{n-1}$  ($1\leq i\leq n$), consisting in forgetting the $i$-coordinate, and
\item the negative connections $\gamma_{i,0}^n:[1]^n\rightarrow [1]^{n-1}$ ($1\leq i\leq n-1$), mapping $(x_1,\ldots,x_n)$ to $
(x_1,\ldots,x_{i-1},\max(x_i,x_{i+1}),x_{i+2},\dots,x_n)$.
\end{itemize}

Adapting Theorem 5.1 of 
Grandis and Mauri in \cite{GM03} (see also \cite{Mal09}) to our case, we have that
every map in the category $\square$ can be factored uniquely as a composite
\[(\partial_{c_1,\epsilon_1}\ldots\partial_{c_r,\epsilon_r})(\gamma_{b_1,0}\ldots\gamma_{b_q,0})(\sigma_{a_1}\ldots\sigma_{a_p})\]
with $1\leq a_1<\ldots<a_p,\; 1\leq b_1<\ldots<b_q,\;  c_1>\ldots>c_r\geq 1$. 

In particular, it factors uniquely as an epimorphism followed by a monomorphism. 
Relying on this factorisation, one can give an alternative presentation of $\square$ by generators (as above) and relations given by  cubical 
identities, as listed in \cite{DKLS} just before Proposition 1.16.

The category of presheaves on $\square$ is called the category of cubical sets and denoted by $c\Set$. The representable presheaves are 
denoted by $\square^n$, and are called the {\sl $n$-cubes}.  

In addition, the 
factorisation of Grandis and Mauri in $\square$ induces 
the existence of the standard form of an $n$-cube $x$ in a cubical set $S$. We recollect here Proposition 1.18, and Corollaries 1.19 and 
1.20, of \cite{DKLS}, where, as usual, ``nondegenerate'' stands for ``not in the image of a degeneracy nor of a connection''.

\begin{prop}\label{P:NDfactorisation} Let $S,T$ be two cubical sets. 
\begin{enumerate}
\item For any $n$-cube $x:\square^n\rightarrow S,$ there exists a unique decomposition $x=y\circ \varphi$, where  $\varphi:
\square^n\rightarrow\square^m$ is an epimorphism and $y : \square^m\to S$ is a nondegenerate $m$-cube.
\item Any map  $\varphi:S\rightarrow T$ in $c\Set$ is determined by its action on nondegenerate cubes.
\item A map  $\varphi:S\rightarrow T$ is a monomorphism if and only if it maps nondegenerate cubes of $S$ to nondegenerate cubes of 
$T$ and does so injectively.
\end{enumerate}
\end{prop}

\begin{defn}\label{order_ncube}A {\it vertex} of a cubical set $S$ is an element of $S_0$ (where $S_0 = S([1]^0)$) and the vertices of 
$\square^n$ are in one-to-one 
correspondence with the $n$-tuples $(a_1,\ldots,a_n)$ of $[1]^n$, or equivalently with the subsets of $\{1,\ldots,n\}$.
In the paper, we will use either point of view, depending on the context.

The order of $[1]^n$ induces an order $\preccurlyeq$ on  the vertices of $\square^n$:
\[(a_1,\ldots,a_n)\preccurlyeq (b_1,\ldots,b_n) \Longleftrightarrow a_i\leq b_i,\; \forall 1\leq i\leq n.\]
It is isomorphic to the subset lattice of $\{1,\ldots,n\}$ via $(a_1,\ldots,a_n)\mapsto \{i\; |\; a_i=1\}$.  Hence it has a least element  
$\alpha=\emptyset$ and a greatest element $\omega=\{1,\ldots,n\}$ (or $\alpha=(0,\ldots,0)$ and $\omega=(1,\ldots,1)$).\\
For $a\preccurlyeq b$,  let  $d(a,b)$ be the cardinal of $b\setminus a$  and let
$\iota^n_{a,b}$  be the face map
$\square^{d(a,b)} \hookrightarrow \square^n$ satisfying $\iota^n_{a,b}(\alpha)=a$ and $\iota^n_{a,b}(\omega)=b$ (see Lemma \ref{L:distance}).
\end{defn}

\begin{ex}If $n=5$, $a=(1,0,0,0,0)$, $b=(1,0,1,0,1)$, then $b\setminus a=\{3,5\}$, and $\iota^5_{a,b}:\square^2\rightarrow \square^5$ is given by
$\iota^5_{a,b}(x,y)=(1,0,x,0,y)$. 
\end{ex}

\begin{lem}\label{L:distance} A map $\varphi:\square^n\rightarrow\square^m$ satisfies $d(\varphi(\alpha),\varphi(\omega))\leq n$. The map $\varphi$ is a monomorphism if and 
only if $d(\varphi(\alpha),\varphi(\omega))= n$, and in this case $\varphi$ is determined by $\varphi(\alpha)$ and $\varphi(\omega)$.
 In particular, if $n=1$, then  $\varphi(\alpha)=\varphi(\omega)$ or $\varphi(\omega)\setminus\varphi(\alpha)=\{i\}$ for some $i$.
\end{lem}

\begin{proof} We decompose $\varphi=u\circ v$ with $v:\square^n\rightarrow\square^p$ a composition of degeneracies and connexions and $u:\square^p\rightarrow\square^m$ 
a composition of faces. We have $p\leq n$ and $v(\alpha)=\alpha, v(\omega)=\omega$.
A composition of faces inserts some $0$ and $1$ at some places and thus lets the distance between two vertices invariant. In particular $d(u(\alpha),u(\omega))=p\leq n$. 
In addition, since degeneracies and connections  always decrease $d(\alpha,\omega)$  strictly, we get that $\varphi$ is a monomorphism if and only if $\varphi$ is a composition of faces, 
if and only if  $d(\varphi(\alpha),\varphi(\omega))= n$. A face is uniquely determined by its value on $\alpha$ and $\omega$, so is a composition of faces. The second part of 
the statement is immediate.\end{proof}

We next recall two  model category structures on cubical sets. The first one, the Grothendieck model structure, models homotopy types and is described by Cisinski 
in \cite{Cin06}, and the second one models $(\infty,1)$-categories and is described in \cite{DKLS}.

We  recall  some useful definitions of \cite[Section 4]{DKLS}. 
\begin{defn}\label{D:boiteinterne}  \hspace{1cm}
\begin{itemize}
\item
The boundary of $\square^n$, that is, the union of all the faces of $\square^n$, is denoted by $\partial\square^n$ and the canonical inclusion by  
$\partial^n:\partial\square^n\rightarrow \square^n$.
\item
The union of all the faces except $\partial_{i,\epsilon}$ is denoted by $\sqcap^n_{i,\epsilon}$, and the inclusion $\sqcap^n_{i,\epsilon}\rightarrow \square^n$ is called an {\it  open box inclusion}. 
\item
Given a face $\partial_{i,\epsilon}$ of $\square^n$, its {\sl critical edge} $e_{i,\epsilon}$ is the unique edge of $\square^n$ that is adjacent to $\partial_{i,\epsilon}$ and 
contains the vertex $\alpha$ or $\omega$ which is not in $\partial_{i,\epsilon}$. Namely, this is the edge between the vertices $(1-\epsilon,\ldots,1-\epsilon)$ and 
$(1-\epsilon,\ldots,1-\epsilon,\epsilon,1-\epsilon,\ldots,1-\epsilon)$, where $\epsilon$ is placed at the i-coordinate. Equivalently, for $\epsilon=1$, this is the edge from 
$\alpha$ to $\{i\}$ and if $\epsilon=0$ this is the edge from $\{1,\ldots,n\}\setminus \{i\}$ to $\omega$.
\item
For $n\geq 2$, quotienting by the critical edge results in the {\sl $(i,\epsilon)$-inner cube} $\widehat{\square}^n_{i,\epsilon}$, the {\sl $(i,\epsilon)$-inner open box} 
$\widehat{\sqcap}^n_{i,\epsilon} $, and the {\sl $(i,\epsilon)$-inner open box inclusion} $h^n_{i,\epsilon} : \widehat{\sqcap}^n_{i,\epsilon} \to \widehat{\square}^n_{i,\epsilon}$.
\item
A (cubical) {\sl Kan fibration} is a map having the right lifting property with respect to all open box inclusions.
\item
A (cubical) {\sl inner fibration} is a map having the right lifting property with respect to all inner open box inclusions.
\item A {\sl cubical quasi-category} is a cubical set $X$ such that $X\rightarrow *$ is an inner fibration.
\end{itemize}
\end{defn}
\begin{thm}[Cisinski, Theorem 1.34 in \cite{DKLS}] The category $c\Set$ carries a cofibrantly generated model structure, refered to as the Grothendieck model structure, in which
\begin{itemize}
\item cofibrations are the monomorphisms, and
\item fibrations are Kan fibrations.
\end{itemize}
\end{thm} 
We next sum up   Theorems 4.2  and 4.16 in \cite{DKLS} for the Joyal model structure.

\begin{thm} \label{Cset-Joyal}
The category $c\Set$ carries a cofibrantly generated model structure, refered to as the Joyal model structure, in which
\begin{itemize}
\item cofibrations are the monomorphisms, and
\item fibrant objects are cubical quasi-categories.
\end{itemize}
Moreover, fibrations between fibrant objects are inner fibrations having the right lifting property with respect to the two endpoint inclusions $j_0:\{0\}\rightarrow K$ and $j_1:\{1\}\to K$, where   
$K$ is the cubical set 
\[\xymatrix{1\ar[r]\ar@{=}[d]&0\ar[d]\ar@{=}[r]&0\ar@{=}[d]\\
1\ar@{=}[r]&1\ar[r]&0}
\]
\end{thm}

We next recall the notion of equivalence and of special open box (see \cite[Section 4]{DKLS}).

\begin{defn}\label{D:special} Let $X$ be a cubical set. 
\begin{itemize}
\item An edge $f:\square^1\rightarrow X$ is an {\sl equivalence} if it factors through the inclusion of the middle edge 
$\square^1\rightarrow K$. 
\item For $n\geq 2, 1\leq i\leq n, \epsilon\in\{0,1\}$, a {\sl special open box} in $X$ is a map ${\sqcap}^n_{i,\epsilon}\rightarrow X$ which sends the critical edge $e_{i,\epsilon}$ to an 
equivalence.
\end{itemize}
\end{defn}
Intuitively, in reference to the above drawing of $K$, the definition of  equivalence says that $f$ has a left and and right inverse (the images of the nondegenerate horizontal edges), 
witnessed such by the images of the  two  2-cubes.

Finally, we collect results of \cite[Sections 5 and 6]{DKLS} on the comparison between cubical and simplicial models.
\begin{defn}
We can construct a monoidal product $\otimes:c\Set\times c\Set\rightarrow c\Set$ by taking the left Kan extension of the monoidal product on $\square$ given 
by $[1]^n\times [1]^m\mapsto [1]^{n+m}$ postcomposed with the Yoneda morphism. In particular $\square^n\otimes\square^m\cong \square^{n+m}.$
\end{defn}
Note that this monoidal product is not symmetric.

\smallskip
In \cite{DKLS} the authors provide four  different but analogous functors $s\Set\rightarrow c\Set$, each of them labelled by a face of the $2$-cube. We choose the one 
labelled by the face $\partial_{2,1}$ (corresponding to  $Q_{R,1}$ in \cite{DKLS}) and  denote it by  $Q$ throughout the paper. It is obtained as a left Kan extension of 
a functor $\Delta\rightarrow c\Set$, that we describe in the following definition.

\begin{defn}\label{D:Q} Let $n\geq 0$. The cubical set $Q^n$ 
is the quotient of the $n$-cube $\square^n$ obtained as the following pushout

\begin{tikzcd}[column sep=100pt]
(\square^0\otimes\square^{n-1}) \bigsqcup (\square^1\otimes\square^{n-2}) \bigsqcup \dots \bigsqcup (\square^{n-1}\otimes\square^0) \arrow{d} \arrow[dr, phantom, "\ulcorner", very near end]
\arrow{r}{(\partial_{1,1}, \partial_{2,1}, \dots, \partial_{n,1})} & \square^n \arrow{d}{\pi_n} \\
\square^{n-1} \bigsqcup \square^{n-2} \bigsqcup \dots \bigsqcup \square^0 \arrow{r} & Q^n 
\end{tikzcd}

The family $(Q^n)_{n\geq 0}$ assembles to a functor $Q:\Delta\rightarrow c\Set$ where faces and degeneracies are defined on the
 cubes $\square^n$ and compatible with the pushout. Namely,
\begin{itemize}
\item the $i$-th face $[n-1]\stackrel{d_i}{\longrightarrow}[n]$ is sent to $\square^{n-1}\stackrel{\partial_{i,0}}{\longrightarrow}\square^n$ if $i>0$ and to  $\partial_{1,1}$ if $i=0$, and
\item the $i$-th degeneracy $[n+1]\stackrel{s_i}{\longrightarrow}[n]$ is sent to $\square^{n+1}\stackrel{\gamma_{i,0}}{\longrightarrow}\square^n$ if $i>0$ and to $\sigma_1$ if $i=0$.
\end{itemize}
\end{defn}

\begin{lem}\label{L:vertofQ} The set of vertices of $Q^n$ is in bijection with the set $\{0,\ldots,n\}$ and the map $\pi_n$ sends $a\subseteq\{1,\ldots,n\}$ to $\sup a$, setting $\sup \emptyset = 0$.
Furthermore, the action of the faces and degeneracies on the vertices of $Q^n$ coincides with the action on the vertices on the simplicial set $\Delta^n$.
\end{lem}
\begin{proof}
Since colimits in $c\Set$ are computed dimensionwise, the set of vertices $(Q^n)_0$ is obtained as the pushout of the diagram above evaluated at $\square^0$. 
We claim that the set $\{0,\ldots,n\}$, together with the map
\[\begin{array}{cccc}
\pi_n:& \mathcal P(\{1,\ldots,n\}) &\rightarrow &\{0,\ldots,n\}\\
& a&\mapsto & \sup a,
\end{array}\]
satisfy the universal property of the pushout. Consider $I_1\subseteq \{1,\ldots,i-1\}$ and $I_2\subseteq \{i+1,\ldots,n\}$. Then $(I_1,I_2)$ is mapped horizontally to 
$I_1\cup\{i\}\cup I_2$ and vertically to $I_2$, hence $I_1\cup\{i\}\cup I_2$ is identified to $I'_1\cup\{i\}\cup I_2$ for any other $I'_1\subseteq \{1,\ldots,i-1\}$. The claim follows easily 
from this observation.
The rest of the statement is also checked easily.
\end{proof}

The left Kan extension of $Q : \Delta \to c\Set$ along the Yoneda morphism is also denoted by $Q:s\Set\to c\Set$ and then admits a right adjoint $\int$ defined as
$(\int{S})_n:=\Hom_{c\Set}(Q^n,S)$.  We have the following Quillen equivalences
(Corollary 6.24 and Proposition 6.25 of \cite{DKLS}).

\begin{thm}\label{Q_Quillen}
The adjunction  $Q : s\Set \rightleftarrows c\Set:\int$  is both a Quillen equivalence 
\begin{itemize}
    \item  between the Joyal model structure on $s\Set$ and the Joyal model structure on $c\Set$, and
    \item  between the Kan-Quillen model structure on $s\Set$ and the Grothendieck model structure on $c\Set$.
\end{itemize}
\end{thm}

\section{Necklaces and paths} \label{necklace-section}

In this section and the following one, we follow closely the steps taken by Dugger and Spivak in \cite{DS11} in order to understand more concretely the simplicial 
rigidification functor. We adapt their approach to define the simplicial rigidification of cubical sets.

\subsection{Necklaces}\label{S:necklace}

Let $\cSdp = \slicecat{\partial\square^1}{c\Set}$ be the category of double pointed cubical sets. Given a cubical set $S$ and two vertices 
$a,b\in S_0$, the notation $S_{a,b}$ stands for the double pointed cubical set corresponding to the morphism $(\partial\square^1 \to S) \in 
\cSdp$ mapping $0$ to $a$ and $1$ to $b$. We refer to Section  \ref{S:wedge} for general constructions. When there is no ambiguity on the double pointing, we omit the 
indices and write $S\in \cSdp$. For example, the cube $\square^n$ is naturally double pointed by $\alpha$ and $\omega$ (see 
Definition \ref{order_ncube}), and if not specified otherwise we will consider this double pointing.

\begin{defn}\hspace{1cm}
\begin{itemize}
\item A {\sl (cubical) necklace} is an object $T$ of $\cSdp$ of  the form  $\square^{n_1} \vee \dots \vee 
\square^{n_k}$, for some sequence $(n_1,\ldots,n_k)$ of positive integers. The double pointing is induced by $\alpha\in\square^{n_1}$ and 
$\omega\in\square^{n_k}$. The empty sequence corresponds to the necklace $T=\square^0$ and it is the unique one satisfying $\alpha=\omega$.
\item For $k\geq 1$, the canonical morphism  $B_i:\square^{n_i} {\to} T$ in $c\Set$ is called the {\sl $i$-th bead} of $T$, so 
that $\id_T=B_1*\ldots*B_k$ (see Definition \ref{D:concat} for the notation).
\item We denote by $\Nec$  the full subcategory of $\cSdp$, whose objects are cubical necklaces. Objects will be identified with sequences $
(n_1,\ldots,n_k)$ of positive integers. Note that if $S$ is a 
necklace and $(a,b)\not=(\alpha,\omega)$, then an object of  the slice category $\slicecat{\Nec}{S_{a,b}}$, i.e., a morphism $T\rightarrow S_{a,b}$ with $T$ a necklace, 
is not a morphism in $\Nec$ since 
the double pointing in $\Nec$ is given by $(\alpha,\omega)$.
\item Given two sequences $(n_1,\ldots,n_k)$ and $(m_1,\ldots,m_l)$, their {\sl concatenation} is the sequence $(n_1,\ldots,n_k,m_1,\ldots,m_l)$.
\item  A {\sl decomposition}  of a non-empty sequence $(n_1,\ldots,n_k)$ in $l$ blocks is a collection $(A_1,\ldots,A_l)$ of non-empty sequences such that their concatenation is $(n_1,\ldots,n_k)$.
\end{itemize}
\end{defn}

\begin{figure}[h]
  \centering
  \[\xymatrix{ & & \ar[dr] & & & \ar[dr] & & \ar[r] \ar@{.>}[dr] & \ar[dr] & &
\\ \ar@{}[r]^{\hphantom{3cm} \alpha} & \cdot \ar[ur] \ar[dr] \ar@{}[rr]|{B_1} & & \ar[r]^{B_2} & \ar[ur] \ar[dr] \ar@{}[rr]|{B_3} & &  \ar[r] \ar[ur] 
\ar[dr] \ar@{}[rrr]|{B_4} & \ar[ur] \ar[dr] & \ar@{.>}[r] & \cdot \ar@{}[r]^-{\omega \hphantom{3cm}} &  
\\ & & \ar[ur] & & & \ar[ur] & & \ar[r] \ar@{.>}[ur] & \ar[ur] & &}\]
  \caption{The necklace associated to the sequence $(2,1,2,3)$}
\end{figure}

The following proposition describes the morphisms in $\Nec$.

\begin{prop}\label{P:Nec_morphisms} \hspace{1cm}

\begin{enumerate} 
\item In the category $\Nec$, a morphism $\varphi$ from $(n_1,\ldots,n_k)$ to $(m)$ decomposes uniquely as $
\varphi=\varphi_1*\ldots*\varphi_k$, where all $\varphi_i:\square^{n_i}\rightarrow \square^m$  are  morphisms in $\square$, and satisfy
\[\begin{cases} \varphi_1(\alpha)=\alpha,\\
 \varphi_i(\omega)=\varphi_{i+1}(\alpha),&\forall 1\leq i\leq k-1,\\
 \varphi_k(\omega)=\omega.
 \end{cases}\]
 In particular $m\leq n_1+\ldots+n_k$.
 \item
 Given a morphism $f:(n_1,\ldots,n_k)\rightarrow(m_1,\ldots,m_l)$ in $\Nec$, there exists a decomposition $(A_1,\ldots,A_l)$ of the sequence $
 (n_1,\ldots,n_k)$ into $l$ parts and 
 morphisms  $f_j:A_j\rightarrow (m_j)$ in $\Nec$ such that $f=f_1\vee\ldots\vee f_l$. This decomposition is unique if, for any $1\leq i\leq k$, the 
 restriction of $f$ to the bead $
 (n_i)$ is not constant.
\end{enumerate}
 \end{prop}
 
 \begin{proof} The first part of the proposition is a direct consequence of the definition of the concatenation. Note that a map from 
 $(n_1,\ldots,n_k)$ to $(m)$ yields a chain $
 \alpha=a_0\preccurlyeq a_1\preccurlyeq\ldots\preccurlyeq a_{k-1}\preccurlyeq a_k=\omega$ in $(\square^m)_0$ with $d(a_{i-1},a_{i})\leq n_i$, 
 by Lemma  \ref{L:distance}. Hence $m=d(\alpha,\omega)\leq n_1+\ldots+n_k$.\\
 Let us prove the second part. For the sake of clarity, we denote by $(\alpha_i,\omega_i)$ the initial and terminal vertices of $\square^{m_i}$.
 For any cubical set $S$ and any $n\geq 0$, we denote by $\sigma:S_0\rightarrow S_n$ the map induced by the unique map 
 $\square^n\rightarrow\square^0$ in $\square$. Let 
 $T=\square^{m_1}\vee\ldots\vee\square^{m_l}$.  By definition of concatenation,  the set of $n$-cubes in the cubical set $T$ is the quotient of the 
 disjoint union of the $n$-cubes of 
 $\square^{m_i}$ by the relation $\sigma(\omega_i)=\sigma(\alpha_{i+1})$ for $1\leq i\leq l-1$. 
  Let $\varphi: (n_1,\ldots,n_k)\rightarrow T$ be a morphism in $\Nec$ and let $\varphi_i:(n_i)\rightarrow (m_1,\ldots,m_l)$ be its components, that is, $\varphi_i$ 
 is an $n_i$-cube of $T$, and 
 $\varphi=\varphi_1*\ldots *\varphi_k$. Since $\varphi_1(\alpha)=\alpha_1$, necessarily $\varphi_1$ is an $n_1$-cube of $\square^{m_1}$. Since $
 \varphi_1(\omega)=\varphi_2(\alpha)$, there are two possibilities: either $\varphi_1(\omega)\not=\omega_1$ and then $\varphi_2$ is an 
 $n_2$-cube of $\square^{m_1}$, or $
 \varphi_1(\omega)=\omega_1$ and $\varphi_2$ is an $n_2$-cube of $\square^{m_2}$. Inductively, we get a decomposition $(A_1,\ldots,A_l)$, where 
 $A_i$ is a
 sequence of consecutive $n_j$  such that $\varphi_{j}$ is an $n_j$-cube of $\square^{m_i}$.  The decomposition is not unique in general. For 
 example,  if above we  had $\varphi_2=\sigma(\omega_1)$, then we could have chosen to keep $n_2$ in $A_1$.  But under the assumption that 
 each $\varphi_i$ is not constant,  we do have uniqueness, since we can identify unambiguously in 
 which component of $T$ it lies.
 \end{proof}

 \begin{ex}\label{E:mornec} (i) There is no morphism from $(2,1,3)$ to $(3,2,1)$: there is a unique decomposition of $(2,1,3)$ into three parts and, 
 by Lemma \ref{L:distance}, there is no morphism in $\Nec$ from $\square^2$ to $\square^3$.\\
 (ii) A morphism  $f:(2,1,3)\rightarrow (2,1)$ in $\Nec$ is either of the form $g\vee h$ with $g:(2,1)\rightarrow (2)$ and $h:(3)\rightarrow (1)$ in $
 \Nec$ (first type), or is $\id_{(2)}\vee(\id_{(1)} *c_\omega)$, where $c_\omega: (3)\to (1)$ is the constant map with value $
 \omega$. Indeed, there are two decompositions of $(2,1,3)$ into two blocks. The decomposition $((2,1),(3))$ 
 gives the first decomposition. The decomposition $((2),(1,3))$  gives  $f=k\vee(l_1*l_2)$, where $k:(2)\to (2) \in\Nec$ is necessarily the identity and 
 $l_1*l_2:(
 1,3)\to (1)$.
 The only maps $l_1:(1)\to (1)$  such that $l_1(\alpha)=\alpha$ are $c_{\alpha}$ and the identity. If $l_1=c_\alpha$, then 
 $(k*c_\omega)\vee l_2: (2,1)\vee (3)\to (2)\vee (1)$ is also a decomposition of $f$ of the first type. If $l_1=\id_{(1)}$, then 
 $l_2=c_\omega$.
 \end{ex}

 \begin{notation}If $T=(n_1,\ldots,n_k)$ is a cubical necklace, then 
$T_0=(\square^{n_1})_0\vee\ldots\vee(\square^{n_k})_0$ is a bounded poset (see  Definition \ref{bounded-poset}). We will also denote the order
 in $T_0$ by $\preccurlyeq$. Moreover, 
 setting $n=n_1+\ldots+n_k$, any monomorphism in $\Nec$ from $T$ to $\square^n$ is a morphism 
 of posets on vertices, which justifies the notation $\preccurlyeq$.
 \end{notation}
 
 The following Lemma is an easy consequence of  Lemma \ref{L:distance}. 

\begin{lem}\label{L:canonical-embedding} Let $T=(n_1,\ldots,n_k)$ be a cubical necklace.
\begin{itemize} 
\item Any monomorphism $\varphi: T\hookrightarrow\square^n$ in $\Nec$ is uniquely determined by a sequence 
$a_0\prec a_1\prec\ldots\prec a_k$ of vertices in $\square^n$ satisfying
$a_0=\alpha, a_k=\omega$ and $d(a_{i-1},a_{i})=n_i$. 
The sequence $\emptyset \prec \{1,\ldots,n_1\}\prec\ldots\prec \{1,\ldots,n_1+\ldots+n_{k-1}\}\prec \{1,\ldots,n\}$ corresponds to an 
embedding $T\hookrightarrow\square^n$ that we will call the {\sl standard embedding}.
\item If $n_1 =\ldots = n_k =1$, then any morphism $\varphi: T\to\square^n$ in $\Nec$ is uniquely determined by a 
sequence $a_0\preccurlyeq a_1\preccurlyeq\ldots\preccurlyeq a_k$ of vertices in $\square^n$ satisfying
$a_0=\alpha, a_k=\omega$ and $d(a_{i-1},a_{i})\leq 1$.
\item  If $n_1 =\ldots = n_k =1$, then any monomorphism $\varphi: T\hookrightarrow\square^n$ in $\Nec$ is uniquely determined by a 
sequence $a_0\prec a_1\prec\ldots\prec a_k$ of vertices in $\square^n$ satisfying
$a_0=\alpha, a_k=\omega$ and $d(a_{i-1},a_{i})=1$.
\end{itemize}
\end{lem}

\begin{defn}\label{D:T_ab}
Let $T=(n_1,\ldots,n_k)$ be a necklace and $a\preccurlyeq b$ be vertices of $T$.
We define $\iota^T_{a,b} : T_{[a,b]} \hookrightarrow T$ in $c\Set$  as follows: 

\begin{itemize}
\item If $a,b\in\square^{n_i}$ then $T_{[a,b]} := \square^{d(a,b)}$ and $\iota^T_{a,b} :=B_i\circ \iota^{n_i}_{a,b}$ (cf. Definition \ref{order_ncube}).
\item If $a\in\square^{n_i}$ and $b\in\square^{n_j}$ with $i<j$, then $T_{[a,b]}:= \square^{d(a,\omega)}\vee\square^{n_{i+1}}
\vee\dots\vee\square^{n_{j-1}}\vee\square^{d(\alpha,b)}$ and $\iota^T_{a,b} := (B_i\circ\iota^{n_i}_{a,\omega}) * B_{i+1}*\ldots* B_{j-1}* (B_j\circ\iota^{n_j}_{\alpha,b})$.
\end{itemize}

Hence $T_{[a,b]}$ is a necklace and $\iota^T_{a,b} : T_{[a,b]} \hookrightarrow T$ is a monomorphism in $c\Set$. We call $T_{[a,b]}$  the {\sl 
subnecklace} of $T$ between $a$ and $b$.
\end{defn}

\begin{rmk}
This is well-defined as $\square^0$ is the unit of the monoidal product $\vee$, so the construction of $T_{[a,b]}$ does not depend on 
the bead chosen for containing $a$ or $b$.
\end{rmk}

\begin{prop}\label{univ_prop_subneck}
Let $T$ be a necklace and $a\preccurlyeq b\in T_0$. The object $\iota^T_{a,b} : T_{[a,b]} \hookrightarrow T$ is terminal in 
$\slicecat{\Nec}{T_{a,b}}$. 
\end{prop}

\begin{proof}
Let $f:X_{\alpha,\omega}\rightarrow T_{a,b}$ be a map in $\cSdp$ with $X$ a necklace. Proceeding like in the proof of  Proposition 
\ref{P:Nec_morphisms}, we get that $f$ factors uniquely through $T_{[a,b]}$ as $f=\iota^T_{a,b}\circ \hat f$ with $\hat f:X_{\alpha,\omega}\rightarrow T_{[a,b]}$ a 
morphism in $\Nec$.
\end{proof}

The next  lemma states properties that will be needed in the proof of Proposition \ref{Ct_with_subneck}.

\begin{lem}\label{L:TNDmono} Let $S$ be a cubical subset of $\square^n$. Let $a\preccurlyeq b$ be two vertices of $S$.
\begin{enumerate}
\item An $m$-cube $x:\square^m\rightarrow S$ is nondegenerate if and only if $x$ is a monomorphism. 
\item A map $f:T\rightarrow S$ with $T$ a necklace is a monomorphism if and only if it is a monomorphism on every bead.
\item Any object  $f:T\rightarrow S_{a,b}$ of $\slicecat{\Nec}{S_{a,b}}$ factors uniquely as $f=\iota(f)\pi(f)$, where $\pi(f):T\to T^f$ is an 
epimorphism in $\Nec$ and  $\iota(f):T^f\to S_{a,b}$ is a monomorphism in $c\Set$.
\end{enumerate}
\end{lem}

\begin{proof} We make use of Proposition \ref{P:NDfactorisation}. \\
$(1)$ If $x$ is nondegenerate in $S$, then $x$ is nondegenerate in $\square^n$, hence a monomorphism, and so is $x$. Conversely, if $x$ is  a monomorphism, 
we can factor uniquely $x=ip$ with $p$ an epimorphism and $i$ a nondegenerate map. In particular $p$ is a monomorphism, thus an isomorphism of cubes, that is, it is the identity. \\
$(2)$ Let $B_i:\square^{n_i}\to T=(n_1,\ldots,n_k)$ be the inclusion of the $i$-th bead of $T$. If $f$ is a monomorphism, then $f B_i$ is. For the converse, we use 
Proposition \ref{P:NDfactorisation}(3). Let $x:\square^m\rightarrow T$ be a nondegenerate $m$-cube of $T$.
Since $x$ factors through one bead $B_i$, we have that $f(x)=(f B_i)(x)$ is nondegenerate. Let  $x,y$ be two nondegenerate cubes in $T$ such that $f(x)=f(y)$ and $x$ and $y$ are not vertices of $T$. 
Assume $x$ factors through $B_i$, and $y$ through $B_j$, with $i<j$. Then $f B_i(\alpha_i)\prec f B_i(\omega_i)\preccurlyeq f B_j(\alpha_j)$. The inequalities hold because $S$ is a 
cubical subset of $\square^n$ and the left one is strict because $f B_i$ is a monomorphism and $x$ is nondegenerate and not a vertex. Hence $f(x)=f(y)$ is not possible. Therefore, 
$x$ and $y$  factor through the same bead, and hence $x=y$ by our assumption.\\
$(3)$ For every bead $B_i$ of $T$, $fB_i$ factors uniquely as $\iota_i(f)\circ \pi_i(f)$ with $\pi_i(f)$ an epimorphism and $\iota_i(f)$ a monomorphism (by (1)). Setting 
$\pi(f)=\pi_1(f)\vee\ldots\vee \pi_k(f)$ and $\iota(f)=\iota_1(f)*\ldots*\iota_k(f)$, we get the desired factorisation (by (2)).  It is unique since $f$ writes uniquely as $f_1*\ldots *f_k$ 
and each $f_i$ factors uniquely.
\end{proof}

\subsection{The path category of a cubical set}
In this section we associate to a cubical set $S$ a category enriched in  prosets (i.e., preordered
sets) $\Cbox_{path}(S) \in \enrichedcat{Proset}$. The idea is that $\Cbox_{path}(S)(a,b)$ has for objects concatenations of
nondegenerate $1$-cubes joining $a$ to $b$ and that the  preorder is induced by the $2$-cubes of $S$.

\begin{notation} For $n\geq 0$, let $\mathrm{I}_n$ be the necklace $(\square^1)^{\vee n}$.
For $n\geq 2$ and $0\leq k\leq n-2$, let $\mathds{I}_{k,n}$ be the concatenation $\mathrm{I}_k\vee\square^2\vee 
\mathrm{I}_{n-2-k}$.
The source and target maps  $s_{k,n},t_{k,n}: \mathrm{I}_n\rightarrow \mathds{I}_{k,n}$ are the morphisms in $\Nec$ defined by
\[s_{k,n}={\id}^{\vee k}\vee (\partial_{1,0}*\partial_{2,1})\vee {\id}^{\vee n-2-k} \text{ and } t_{k,n}={\id}^{\vee k}\vee (\partial_{2,0}*\partial_{1,1})\vee {\id}^{\vee n-2-k}\]
as presented in the following figure:
    \[\xymatrix@R=6pt{
& & & & & \ar[dr]^{B_{k+2}} & & &  &
\\ \mathrm{I}_n \ar@{}[r]|{=} \ar@/_0.4cm/[dd]_{t_{k,n}} & \ar[r]^{B_1} & \ar@{}[r]|{\cdots} & \ar[r]^{B_k} & \ar[ur]^{B_{k+1}} & &  
\ar[r]^{B_{k+3}} & \ar@{}[r]|{\cdots} & \ar[r]^{B_n} &    
\\ & & & & & \ar[dr] & & &  &
\\ \mathds{I}_{k,n} \ar@{}[r]|{=} & \ar[r]^{B_1} & \ar@{}[r]|{\cdots} & \ar[r]^{B_k} & \ar[ur]^{\mathtt{1}} \ar[dr]_{\mathtt{2}} \ar@{}[rr]|{B_{k+1}} & 
&  \ar[r]^{B_{k+2}} & \ar@{}[r]|{\cdots} & \ar[r]^{B_{n-1}} &
\\ & & & & & \ar[ur] & & & &
\\ \mathrm{I}_n \ar@{}[r]|{=} \ar@/_-0.4cm/[uu]^{s_{k,n}} &\ar[r]_{B_1} & \ar@{}[r]|{\cdots} & \ar[r]_{B_k} & \ar[dr]_{B_{k+1}} & &  
\ar[r]_{B_{k+3}} & \ar@{}[r]|{\cdots} & \ar[r]_{B_n} & 
\\ & & & & & \ar[ur]_{B_{k+2}} & & &  &   }\]
  \end{notation}
\begin{defn} \label{path-cat-hom}
Let $S_{a,b}\in\cSdp$.
The set $\Cbox_{path}(S)(a,b)$ of {\sl paths joining $a$ to $b$} is defined as
\[\Cbox_{path}(S)(a,b)=\bigcup_n\Hom(\mathrm{I}_n,S_{a,b})/\sim\]
where $\sim$ is the equivalence relation generated by  $\gamma\sim\gamma'$ if there exists a factorisation in $\cSdp$
    \[\begin{tikzcd}[row sep=14pt]
    \mathrm{I}_n\ar[r, "\gamma"]\ar[d, dashed] & S_{a,b}\\
    \mathrm{I}_m\ar[ur, "\gamma'", swap]
    \end{tikzcd}\]
For  $\gamma:\mathrm{I}_n \to S_{a,b}$ and $\gamma':\mathrm{I}_n \to S_{a,b}$, we write $[\gamma] \leadsto [\gamma']$, if there exists $0\leq k \leq n-2$ and a factorisation in $\cSdp$
    \[
    \begin{tikzcd}[row sep=8pt]
    \mathrm{I}_n \ar[dr, "s_{k,n}", swap] \ar[drr, "\gamma"] &  & \\
    & \mathds{I}_{k,n} \ar[r, dashed] & S_{a,b} \\
    \mathrm{I}_n \ar[ur, "t_{k,n}"] \ar[urr, "\gamma'", swap] &  & \\
    \end{tikzcd}
    \]
 We then define the preorder structure $\Cbox_{path}(S)(a,b)$  as the reflexive transitive closure of $\leadsto$, that we also denote by $\leadsto$.
    \end{defn}

Next proposition lifts the definition at the level of categories enriched in preordered sets, that is, $Proset$-categories.

\begin{prop} \label{Cpath-cat}
Any cubical set $S$ gives rise to a  $Proset$-category  $\Cbox_{path}(S)$ whose objects are the vertices of $S$ and whose homsets are given by Definition \ref{path-cat-hom}, 
with composition given by concatenation of paths. In addition the assignment $S\mapsto \Cbox_{path}(S)$ upgrades to a functor $\Cbox_{path} : c\Set \to \enrichedcat{Proset}$.
\end{prop}

\begin{proof}
For $\gamma:\mathrm{I}_n\to S_{a,b}$ and $\beta:\mathrm{I}_m\to S_{b,c}$, the  class of the concatenation 
$\gamma*\beta: \mathrm{I}_{n+m} = \mathrm{I}_n \vee \mathrm{I}_m {\to} S_{a,c}$ does not depend on the choice of the representatives 
$\gamma$ and $\beta$. This defines  a composition $\Cbox_{path}(S)(b,c)\times\Cbox_{path}(S)(a,b)\to \Cbox_{path}(S)(a,c)$ by 
$([\beta],[\gamma])\mapsto [\gamma*\beta]$. Similarly if $[\gamma]\leadsto[\gamma']$ and $[\beta]\leadsto[\beta']$ then 
$[\gamma*\beta]\leadsto[\gamma'*\beta']$, and everything is functorial in $S$.
\end{proof}

\subsection{The path category of a necklace}\label{S:pathcatneck} Let $A$ be a totally ordered set with $k$ elements.  An element in the set of bijections $\Sigma_A$ of 
$A$ is represented by a sequence $(a_1,\ldots,a_k)$ such that $\{a_1,\ldots,a_k\}=A$. We consider the (reverse right) weak Bruhat order on $\Sigma_A$, that is, the order generated by 
\[(a_1,\ldots,a_i,a_{i+1},\ldots,a_k)\leadsto_B (a_1,\ldots,a_{i+1},a_i,\ldots,a_k), \text{ if } 1\leq i<k\; \text{and } a_i>a_{i+1}.\] 
For example, for $A=\{1,2,3\}$  the Hasse diagram of $\leadsto_B$ is given by
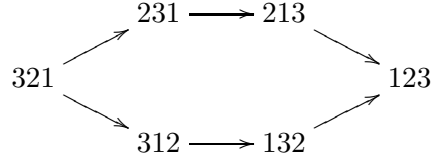
\[\xymatrix@R=1em{ & 231\ar[r]  &213\ar[dr] &\\
321\ar[ur] \ar[dr] & & &123\\
& 312\ar[r]  &132\ar[ur] &}\]
Note that, given two disjoint subsets $A$ and $B$  of $\{1,\ldots,n\}$,  the concatenation $*:\Sigma_A\times \Sigma_B\rightarrow \Sigma_{A\sqcup B}$ of sequences is a 
map of posets (for the order $\leadsto_B$), where the total orders on $A$, $B$ and $A\sqcup B$ are induced by that of \{1,\ldots,n\}. We refer to the book  \cite{Bj84} by 
Bj\"orner for more on orders on Coxeter groups.

\begin{lem} \label{unique-representative}
 For all $n,a,b$, each element in $\Cbox_{path}(\square^n)(a,b)$ has a unique 
representative 
$\gamma$ which is a monomorphism, corresponding to
 a sequence 
$a_0=a\prec a_1\prec\ldots\prec a_{l}=b$, with $d(a_i,a_{i+1})=1$. The same holds replacing $\square^n$ with any cubical subset $S$ of 
$\square^n$.
\end{lem}
\begin{proof}
Let $[\gamma']$ be an element of $\Cbox_{path}(\square^n)(a,b)$.
By  Lemma \ref{L:canonical-embedding},  $\gamma'$ corresponds to some sequence $s'=(a'_0=a\preccurlyeq x_1\preccurlyeq\ldots\preccurlyeq a'_{k'}=b)$
such that $d(a'_i,a'_{i+1})\leq 1$. We claim that the desired $\gamma$ is the monomorphism corresponding to the sequence 
$\kappa(s')=(a_0=a\prec a_1\prec\ldots\prec a_{l}=b)$ obtained by eliminating the repetitions in the sequence $s'$.  This is a consequence of the following two easy facts: 
(i) for $\gamma$ as just defined, we have $[\gamma]\sim[\gamma']$, and (ii)
 if $[\gamma'_1]\sim [\gamma'_2]$, with corresponding sequences $s'_1,s'_2$, then $\kappa(s'_1)=\kappa(s'_2)$.\\
 The last part of the statement follows from the observation that if $s'$ above lies in $S$, then so does $\kappa(s')$.
\end{proof}

\begin{prop}\label{P:pathcubes} For every pair of vertices $a\preccurlyeq b$  in $\square^n$, there is an isomorphism of  preordered sets
\[ \Cbox_{path}(\square^n)(a,b)\rightarrow \Sigma_{b\setminus a},\]
compatible with concatenation. As a consequence, the  preorder $\leadsto$ on paths of a cube is a partial order, isomorphic to the weak 
Bruhat order $\leadsto_B$ on the symmetric group.

\end{prop}

\begin{proof}  Assume $a\preccurlyeq b$.  With the notation of  Lemma \ref{unique-representative}, we can associate with each 
$[\gamma]\in\Cbox_{path}(\square^n)(a,b)$ a sequence 
$a_0=a\prec a_1\prec\ldots\prec a_l=b$, with $d(a_i,a_{i+1})=1$. We denote by $x_i$ the unique element in $a_{i+1}\setminus a_i=\{x_i\}$, so that
$\{x_1,\ldots,
x_l\}=b\setminus a$. Then the map
$\Psi: \Cbox_{path}(\square^n)(a,b)\rightarrow \Sigma_{b\setminus a}$ sending $[\gamma]$ to the sequence $(x_1,\ldots,x_l)$ in $
\Sigma_{b\setminus a}$ is  well defined and bijective. \\
Let $f: \mathds{I}_{k,m}\rightarrow \square^n_{a,b}$ in $\cSdp$, witnessing $[f\circ s_{k,m}] \leadsto [f\circ t_{k,m}]$. Let
\[\begin{cases}
a_0\preccurlyeq\ldots\preccurlyeq a_k\preccurlyeq a_{k+1}\preccurlyeq a_{k+2}\preccurlyeq \ldots\preccurlyeq a_m\\
  a_0\preccurlyeq\ldots\preccurlyeq a_k\preccurlyeq a'_{k+1}\preccurlyeq a_{k+2}\preccurlyeq \ldots\preccurlyeq a_m
\end{cases}\]
be the sequences corresponding to $f\circ s_{k,m}$ and $f\circ t_{k,m}$, respectively.
If $d(a_k,a_{k+2})=2$, then there exists $u<v$ such that $a_{k+2}\setminus a_k=\{u,v\}$, $a_{k+1}\setminus a_k=\{v\}$ and 
$a'_{k+1}\setminus a_k=\{u\}$, hence $\Psi(f\circ s_{k,m})\leadsto_B \Psi(f\circ t_{k,m})$. If $d(a_k,a_{k+2})<2$, then 
$$[f\circ s_{k,m}]=[\gamma]=[f\circ 
t_{k,m}],$$
where $\gamma$ corresponds to $a_0\preccurlyeq\ldots\preccurlyeq a_k\prec a_{k+2}\preccurlyeq \ldots\preccurlyeq a_m$. Hence $\Psi$ is a morphism of  preordered sets. 
Similarly, and even more straightforwardly, we see that $\Psi^{-1}$ is also a morphism of prosets, which in particular implies that $\Cbox_{path}(\square^n)(a,b)$ is 
 a poset. 
\end{proof}

We also observe that if $a\not\preccurlyeq b$, there is no morphism $\mathrm{I}_m\rightarrow (\square^n)_{a,b}$, and that the constant path is the unique path in 
 $\Cbox_{path}(\square^n)(a,a)$.
Hence $\Cbox_{path}(\square^n)$ is a $P$-shaped poset-category, with $P$ the subset lattice of $\{1,\ldots,n\}$. We refer to  Appendix 
\ref{S:Pshapedcat} for this notion, and for the description of the   concatenation product $\vee$ on such categories used in the following proposition.

\begin{cor}\label{C:pathneck} Let $T=\square^{n_1}\vee\ldots\vee\square^{n_k}$  be a necklace.
We have:
\begin{enumerate}
\item\label{cor_neck_1} For $a\preccurlyeq b\in T_0$, the inclusion $T_{[a,b]} \subseteq T_{a,b}$ induces an isomorphism of posets $\Cbox_{path}(T)(a,b) \cong \Cbox_{path}
(T_{[a,b]})(\alpha,\omega)$.
\item\label{cor_neck_2} If $T = U\vee V$, the composition in the poset category $\Cbox_{path}(T)$ provides a morphism 
$\Cbox_{path}(V)(\alpha_V,\omega_V)\times\Cbox_{path}(U) (\alpha_U,\omega_U)\to \Cbox_{path}(T)(\alpha_T,\omega_T)$, which is an isomorphism of posets.
\item  $\Cbox_{path}(T)$ is a poset-category and we have an isomorphism of poset-categories
\[ \Cbox_{path}(T)\cong\Cbox_{path}(\square^{n_1})\vee\ldots\vee \Cbox_{path}(\square^{n_k}).\]
\end{enumerate}
 \end{cor}

\begin{proof} $(1)$ 
By Definition \ref{D:T_ab} and Proposition 
\ref{univ_prop_subneck}, a path $\gamma$ 
joining $a$ to $b$ in $T$ is equivalent to a morphism $\mathrm{I}_n\rightarrow T_{[a,b]}$, where $T_{[a,b]}$ is a necklace. By Proposition \ref{univ_prop_subneck}, 
any map $\mathds{I}_{k,n} \to T_{a,b}$ factorises through $T_{[a,b]}$, hence the result. \\
$(2)$ Let us prove that the morphism of prosets  induced by composition/concatenation
\[\Cbox_{path}(\square^{n_2})(\alpha_2,\omega_2)\times\Cbox_{path}(\square^{n_1}) (\alpha_1,\omega_1)\to \Cbox_{path}(\square^{n_1}\vee\square^{n_2})
(\alpha_1,\omega_2)\]
is an isomorphism of prosets. By Lemma \ref{unique-representative}, and viewing $\square^{n_1}\vee\square^{n_2}$ as a cubical subset of 
$\square^{n_1+n_2}$ via the standard embedding, any element in the right hand side admits a unique representative $\gamma:\mathrm{I}_l\to \square^{n_1}\vee\square^{n_2}$ 
which is a monomorphism. 

Since $\gamma$ preserves  $\alpha$ and $\omega$, we have $l=n_1+n_2$ and thus $\gamma=\gamma_1\vee\gamma_2$ 
is the unique decomposition  provided by Proposition \ref{P:Nec_morphisms}. Hence the morphism is a bijection. We have to prove that 
$[\gamma]\leadsto [\gamma']$ implies $[\gamma_1]\leadsto [\gamma'_1]$ and $[\gamma_2]\leadsto [\gamma'_2].$
Any $f:\mathds{I}_{k,m}\rightarrow \square^{n_1}\vee\square^{m_1}$ factors as $f=f_1\vee f_2$, where either $f_1$  or $f_2$ is a path. 
It implies that if $[\gamma]\leadsto [\gamma']$ then either $[\gamma_1]\leadsto [\gamma'_1]$ and $[\gamma_2]=[\gamma'_2]$, 
or $[\gamma_2]\leadsto [\gamma'_2]$ and $[\gamma_1]=[\gamma'_1]$. In conclusion,  (2) holds, since the left hand side of the morphism is 
a poset. \\
$(3)$ It is clear that this generalises to any  finite wedge product of cubes. In particular, if 
$T_{[a,b]}:= \square^{d(a,\omega_i)}\vee\square^{n_{i+1}}\vee\dots\vee\square^{n_{j-1}}\vee\square^{d(\alpha_j,b)}$, 
we have
$$\begin{array}{llll}
\Cbox_{path}(T)(a,b)\\
\; = \Cbox_{path}(T_{[a,b]})(\alpha,\omega) & \mbox{(by (1))}\\
\; \cong   \Cbox_{path}(\square^{d(a,\omega_i)})(\alpha_i,\omega_i) \times \dots \times\Cbox_{path}(\square^{d(\alpha_j,b)})(\alpha_j,
\omega_j) & 
  \mbox{(by (2))}\\
 \; =  \Cbox_{path}(\square^{n_i})(a,\omega_i)\times\dots\
 \times\Cbox_{path}(\square^{n_j})(\alpha_j,b) & \mbox{(by (1))}\\
 \; \cong  (\Cbox_{path}(\square^{n_i})\vee\dots\vee\Cbox_{path}(\square^{n_j}))(a,b)  & \mbox{(by Proposition \ref{P:V-concat})}.
 \end{array}
$$  
Finally, we note that having established this isomorphism a fortiori implies that $\Cbox_{path}(T)$ is poset-enriched, since all 
$\Cbox_{path}(\square^{n_i})$ are.
\end{proof}

Applying the nerve functor from Proset-categories to simplicial categories, we get the functor $N\circ\Cbox_{path}:c\Set\rightarrow s\Cat$. 
Unfortunately it is not cocontinuous, as we  show in Example  \ref{(C-path-not-continuous}, so that it cannot serve as a left functor in a Quillen equivalence. The next section is devoted to build  such a functor and to study its properties.

\begin{ex} \label{(C-path-not-continuous}
In this example, all simplicial categories involved have only one (possibly) non-trivial mapping space, and hence reduce to simplicial sets.
Consider the quotient $\tilde\Box^2$ of $\Box^2$ obtained by colllapsing the edges between $(0,0)$ and $(0,1)$, and between  $(1,0)$ and $(1,1)$, and 
consider the pushout $X$ of the two horizontal inclusions $\partial_{2,0},\partial_{2,1}:\Box^1\rightarrow\tilde\Box^2$. The cubical set $X$ can be represented as follows
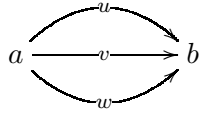
\[\xymatrix{a\ar@/^1.5pc/[rr]|u\ar[rr]|v\ar@/_1.5pc/[rr]|w &&b}\]
with two nondegenerate 2-cubes inducing $u \leadsto v\leadsto w$ in $\Cbox_{path}(X)$. It follows that $N(\Cbox_{path}(X))$ is not 1-skeletal.
However $N(\Cbox_{path}(\Box^1))$ and $N(\Cbox_{path}(\tilde\Box^2))$ are 1-skeletal. But a pushout of 1-skeletal simplicial sets is 1-skeletal, so the pushout of 
$N(\Cbox_{path}(\partial_{2,0}))$ and $N(\Cbox_{path}(\partial_{2,1}))$ cannot be $N(\Cbox_{path}(X))$.
The point of this counter-example is that taking a colimit in $c\Set$ may result in merging of orders: $X$ sees $u\leadsto v\leadsto w$,  while each copy of $\tilde\Box^2$ 
sees only $u\leadsto v$ or $v\leadsto w$. By applying the functor $N\circ \Cbox_{Path}$ before the colimit functor, we lose this piece of magic!
\end{ex}

\section{The  rigidification functor \texorpdfstring{$\Cbox$}{}} \label{rigidification-section}

In this section, we define rigidification  as a left Kan extension of the restriction of $N\circ \Cbox_{Path}$ to the cubes $\Box^n$, and provide concrete descriptions of its  
simplicial homsets, making an essential use of necklaces (see Remark \ref{the-why-of-necklaces}).

\subsection{Definition of the rigidification}
The rigidification functor is defined as the left Kan extension along the Yoneda functor $Y:\square \to c\Set$ of the composition \[\square \stackrel{Y}\to c\Set 
\stackrel{\Cbox_{path}}\to \enrichedcat{Proset} \stackrel{N}\to s\Cat\]
By usual means, we obtain an adjunction $\Cbox \colon c\Set\rightleftarrows s\Cat \colon \Nbox$. The simplicial category $\Cbox(S)$ is obtained as the colimit 
over the category of elements of $S$ of  
some $\Cbox(\square^n)$.

\begin{lem}\label{L:cocontinuous}
For every cubical set $S$, the set of objects of the simplicial category $\Cbox(S)$ is in bijection with $S_0$, and thus will be identified with it.
\end{lem}
\begin{proof}
By Remark \ref{R:objets}, the functor $s\Cat \stackrel{Ob}{\to} \Set$ is cocontinuous. We conclude since the statement holds on cubes  by definition (cf. Proposition \ref{Cpath-cat}).
\end{proof}

\begin{notation} The previous lemma implies that the rigidification functor lifts to a functor $\Cbox : \cSdp \to \sCdp$.  We denote by $\Cbox_t$  
the functor from $\cSdp$ to $s\Set$ defined on objects by 
\[\Cbox_t(S_{a,b}) = \Cbox(S)(a,b),\ \forall a,b\in S_0.\]
\end{notation}

\begin{lem}\label{Ct_cubes_contractible}
The space $\Cbox_t(\square^n)$ is contractible.
\end{lem}
\begin{proof}
We have $\Cbox_t(\square^n)=\Cbox(\square^n)(\alpha,\omega)\cong N(\Sigma_{\{1,\ldots,n\}})$ by Proposition \ref{P:pathcubes} with the 
weak Bruhat order on $\Sigma_{\{1,\ldots,n\}}$. The latter is a bounded poset (see Definition \ref{bounded-poset}), hence its nerve is contractible.
\end{proof}

A direct application of Corollaries \ref{C:pathneck} and  \ref{C:nerveconcat} is the following theorem.

\begin{thm}\label{rigid_neck}
For all necklaces $T$, there is an isomorphism of simplicial categories
\[\Cbox(T) \cong N(\Cbox_{path}(T)).\]
In addition, if $T = U\vee V$, the composition in the simplicial category $\Cbox(T)$ provides a morphism $\Cbox_t(V)\times\Cbox_t(U) \to \Cbox_t(T)$, which is an 
isomorphism of simplicial sets.
\end{thm}

\begin{cor}\label{Ct_nacklaces_contractible}
Let $T$ be a necklace.  For every   $a\preccurlyeq b\in T_0$, the simplicial set $\Cbox(T)(a,b)$ is contractible, in particular the simplicial set $\Cbox_t(T)$ is contractible.
\end{cor}
\begin{proof} By Corollary \ref{C:pathneck},   $\Cbox_{path}(T)(a,b)$ is a product of bounded posets, hence its nerve is contractible. We conclude using Theorem \ref{rigid_neck}.
\end{proof}

\subsection{Computing the rigidification functor}

The construction by left Kan extension gives us a way to express $\Cbox(S)$ as a colimit in $s\Cat$, which is difficult to compute. In this section, 
we use necklaces as ``paths of higher dimension'' to obtain a handy way to compute $\Cbox(S)$. Indeed, we follow step by step the techniques 
developed by Dugger and Spivak in \cite[Proposition 4.3]{DS11}.

Each object $f:T \to S_{a,b}$ in $\slicecat{\Nec}{S_{a,b}}$ induces a morphism $\Cbox_t(T) \to \Cbox_t(S_{a,b})$ in $s\Set$ and this construction gives a 
morphism $\colim_{\slicecat{\Nec}{S_{a,b}}}\Cbox_t(T) \to \Cbox_t(S_{a,b})$ in $s\Set$. We prove that it is an isomorphism,  so that  $\Cbox(S)(a,b)=\C_t(S_{a,b})$ 
is computed as a colimit in $s\Set$.

\begin{notation} Let $S_{a,b}\in\cSdp$.
\begin{itemize}
    \item We set $E_t(S_{a,b}):=\colim_{\slicecat{\Nec}{S_{a,b}}}\Cbox_t(T)$. 
        \item Let $(\alpha_t)_{S_{a,b}}: E_t(S_{a,b}) \to \Cbox_t(S_{a,b})$ be the structure map from the colimit to the cocone $\Cbox_t(S_{a,b})$.
          \end{itemize}
 The following facts are left to the reader.
 \begin{itemize}
 \item The definition of $E_t$ is functorial, hence defines a functor $E_t : \cSdp \to s\Set$.
 \item The morphisms $(\alpha_t)_{S_{a,b}}$ in $s\Set$ form a natural transformation $\alpha_t : E_t \Rightarrow \Cbox_t$.
 \end{itemize}
 \end{notation}

Our goal is to prove that $\alpha_t$ is a natural isomorphism. In fact, we will prove that $E_t$ can be upgraded to a functor $E : c\Set \to s\Cat$ that is naturally 
isomorphic to $\Cbox$.

\begin{prop}
There exists a functor $E : c\Set \to s\Cat$ and a natural transformation $\alpha : E\Rightarrow\Cbox$ such that
\begin{itemize}
\item $\Ob(E(S)) = S_0$,
\item $E(S)(a,b) = E_t(S_{a,b})$,
\item $\alpha_S$ is the dentity on objects and $\alpha_S(a,b) = (\alpha_t)_{S_{a,b}} : E(S)(a,b)\to\Cbox(S)(a,b)$.
\end{itemize} 
\end{prop}
\begin{proof}
We take $*\to S_{a,a} \in \slicecat{\Nec}{S_{a,a}}$  as identity morphism ${\id}_a\in E(S)(a,a)$.
The composition in the simplicial category $E(S)$ is defined to be the composite  featured as the left  arrow in the following diagram:
    \[\begin{tikzcd}
    E_t(S_{b,c})\times E_t(S_{a,b}) = \colim\limits_{V\to S_{b,c}}\Cbox_t(V) \times \colim\limits_{U\to S_{a,b}}\Cbox_t(U) 
    \ar[d, dashed, "\circ_{E}"] & \colim\limits_{\substack{V\to S_{b,c} \\ U\to S_{a,b}}}(\Cbox_t(V)\times \Cbox_t(U)) 
    \ar[l, swap, "\cong"] \ar[d, "\text{Th. } \ref{rigid_neck} ", "\cong"'] \\ 
    E_t(S_{a,c}) = \colim\limits_{T\to S_{a,c}}\Cbox_t(T) & \colim\limits_{\substack{V\to S_{b,c} \\ U\to S_{a,b}}}\Cbox_t(U\vee V) \ar[l]
    \end{tikzcd}\]
 where the top arrow is invertible, as $\times$ is cocontinuous in $s\Set$.
Then the monoidal structure of $(\cSdp,\vee,*)$ ensures that $E(S)$ with identities and composition as above is a simplicial category. The functoriality of $E$ comes from that of $E_t$.

 Let us check that the $(\alpha_t)_{S_{a,b}}$ induce an enriched functor $\alpha_S : E(S)\to\Cbox(S)$, which is the identity on objects. 
 We have to prove that the following diagram commutes:
       \[\xymatrix@C=0.6em{
    \colim\limits_{{V\to S_{b,c}}}\Cbox_t(V) \times \colim\limits_{U\to S_{a,b}}\Cbox_t(U) \ar[d]_{(\alpha_t)_{S_{b,c}}\times(\alpha_t)_{S_{a,b}}} 
    &
     \colim\limits_{\substack{V\to S_{b,c} \\ U\to S_{a,b}}} (\Cbox_t(V)\times \Cbox_t(U)) \ar[l]_-{\cong} 
     \ar[r]^-{\cong}\ar[dl] & \colim\limits_{\substack{V\to S_{b,c} \\ U\to S_{a,b}}}\Cbox_t(U\vee V) \ar[r] \ar[dr]
     & \colim\limits_{T\to S_{a,c}}\Cbox_t(T) \ar[d]^{(\alpha_t)_{S_{a,c}}} \\ 
    \Cbox_t(S_{b,c})\times\Cbox_t(S_{a,b}) \ar[rrr] & & & \Cbox_t(S_{a,c})
}\]  
    It suffices to notice that for all $V\to S_{b,c}$ and $U\to S_{a,b}$, $\Cbox(U\vee V)\to \Cbox(S)$ is a simplicially enriched functor and so the following square commutes:
       \[\begin{tikzcd}
    \Cbox(U\vee V)(\alpha_V,\omega_V)\times\Cbox(U\vee V)(\alpha_U,\omega_U) \ar[r] \ar[d] & \Cbox(U\vee V)(\alpha_U,\omega_V) \ar[d] \\ 
    \Cbox(S)(b,c)\times\Cbox(S)(a,b) \ar[r] & \Cbox(S)(a,c)
    \end{tikzcd}\]
   Theorem \ref{rigid_neck} applied to our case gives  $\Cbox(U\vee V)(\alpha_V,\omega_V)= \Cbox(V)(\alpha_V,\omega_V)$ and
  $ \Cbox(U\vee V)(\alpha_U,\omega_U)=\Cbox(U)(\alpha_U,\omega_U)$, so that the diagram writes:
     \[\begin{tikzcd}
    \Cbox_t(V)\times\Cbox_t(U) \ar[r] \ar[d] & \Cbox_t(U\vee V) \ar[d] \\ \Cbox_t(S_{b,c})\times\Cbox_t(S_{a,b}) \ar[r] & \Cbox_t{(S_{a,c})}
    \end{tikzcd}\]
  and we conclude by universality of colimits. The naturality of $\alpha_S : E(S)\to\Cbox(S)$ comes from the naturality of $\alpha_t$.
\end{proof}

\begin{prop}\label{alpha_nat_iso}
The natural transformation $\alpha : E \Rightarrow \Cbox$ is a natural isomorphism.
\end{prop}
\begin{proof}
Let $S\in c\Set$.  We know that $\alpha_S$ is the identity on objects.  
Assume first that $S=T$ is a necklace. Let $a,b\in T_0$. If $a\preccurlyeq b$, by Proposition \ref{univ_prop_subneck}, $\iota^T_{a,b} : T_{[a,b]} \hookrightarrow T$ is terminal in 
$\slicecat{\Nec}{T_{a,b}}$, 
hence  $E(T)(a,b) \cong \Cbox_t(T_{[a,b]})=\Cbox(T)(a,b)$, and the isomorphism is precisely induced by $\alpha_T$. If $a\not\preccurlyeq b$, 
then both categories are empty, hence the result holds for necklaces. 

We prove that for all vertices $a,b$ in  $S$, the morphism $\alpha_S(a,b) : E(S)(a,b)\to\Cbox(S)(a,b)$ is an isomophism of simplicial sets, by providing 
an inverse $\beta_S(a,b)$. Recall that $\alpha_S(a,b) = \alpha_t(S_{a,b}):E_t(S_{a,b}) \to \Cbox_t(S_{a,b})$ is  the (unique) map from the colimit to 
the cocone $\Cbox_t(S_{a,b})$. Define  $(\beta_t)_{S_{a,b}} : \Cbox_t(S_{a,b})\to E_t(S_{a,b})$ as the composite below:
\[\begin{tikzcd}
    E_t(S_{a,b})  & (\colim\limits_{\square^k\to S}E(\square^k))(a,b) \ar[l] \ar[d, "\cong"',"\square^k \text{ is a necklace }"] \\
    \Cbox_t(S_{a,b}) \ar[u, dashed, "(\beta_t)_{S_{a,b}}"]&  (\colim\limits_{\square^k\to S}\Cbox(\square^k))(a,b) \ar[l, swap, "\cong"]
    \end{tikzcd}\]
 By naturality of all the maps involved in the diagram,  the family $(\beta_t)_{S_{a,b}}$ assembles to a natural transformation $\beta_t : \Cbox_t \Rightarrow E_t$, 
 giving rise to $\beta:\Cbox \Rightarrow E$.

 We show, in this order, that $\beta$ is a right inverse, and a left inverse of $\alpha$. To show the former, it is enough to show $\alpha_S\circ \beta_S\circ j_f= j_f$, 
 for all  $f:\Box^k\rightarrow S_{a,b}$, where $j_f$ is the characteristic map $E(\Box^k)\rightarrow \colim\limits_{\square^k\to S} E(\Box^k)$, and 
 where we identify $ \Cbox(S)$ with $\colim\limits_{\square^k\to S} E(\Box^k)$.  Indeed, we have
 $$\begin{array}{lllll}
 \alpha_S\circ \beta_S\circ j_f & = & \alpha_S \circ E(f) & \text{by definition of }\beta\\
 & = & \Cbox(f)\circ \alpha_{\Box^k} & \text{by naturality of }\alpha\\
 & = & j_f & \text{by the identification above}.
 \end{array}$$
 
We prove now  that $(\beta_t)_{S_{a,b}} \circ (\alpha_t)_{S_{a,b}}$ is the identity in a similar way. By the universal property of the colimit, 
it is enough to prove $(\beta_t)_{S_{a,b}}\circ(\alpha_t)_{S_{a,b}}\circ i_f = i_f$ for all $T \stackrel{f}\to S_{a,b}$., where $i_f : \Cbox_t(T)\to E_t(S_{a,b})$ is the 
characteristic morphism.
This comes from the commutative diagram:
    \[\begin{tikzcd}
    &  E_t(S_{a,b}) \ar[d, "(\alpha_t)_{S_{a,b}}"] \\
    \Cbox_t(T) \ar[r, "\Cbox_t(f)"] \ar[dd, swap, bend right=50, "\id"] \ar[ur, dashed, "i_f"] \ar[d, "\cong"', "(\beta_t)_T"] & \Cbox_t(S_{a,b}) \ar[d, "(\beta_t)_{S_{a,b}}"] \\
    E_t(T) \ar[d, "\cong"', "(\alpha_t)_T"] \ar[r, "E_t(f)"] & E_t(S_{a,b}) \\
    \Cbox_t(T) \ar[ur, swap, dashed, "i_f"]
    \end{tikzcd}\]
    (the above triangle commutes by definition of $\alpha_t$, the middle square commutes by naturality of $\beta_t$, and the bottom triangle commutes by definition of $E_t(f)$).
\end{proof}
As a direct corollary, we get the main theorem of the section.
\begin{thm}\label{T:Ct_with_neck}
Let $S$ be a cubical set and $a,b\in S_0$. We have the following isomorphism of simplicial sets
 \[\Cbox_t(S_{a,b})=\Cbox(S)(a,b) \cong \colim\limits_{(T\to S_{a,b})\;\in\; \slicecat{\Nec}{S_{a,b}}}\Cbox_t(T).\]
\end{thm}

In the following remark, we point out that necklaces were instrumental in getting the characterisation given in the previous theorem. 
\begin{rmk} \label{the-why-of-necklaces}
We {\em do not} have an isomorphism between  $\Cbox(S)(a,b)$ and the colimit above restricted to cubes. This already fails when $S$ is a necklace.  
Consider $S=(1,1)$ and $(a,b)=(\alpha,\omega)$. There is no morphism from a cube to $S_{\alpha,\omega}$ in $\cSdp$, whereas there is one from the 
necklace $S$ to itself (the identity morphism). Hence the restricted colimit is empty, while $\Cbox(S)(a,b)$ is  not.
\end{rmk}

\begin{rmk} It can be shown, using the result above, that $\pi_0(\Cbox(S))\cong\pi_0(\Cbox_{path}(S))$ for any cubical set $S$. \end{rmk}

\subsection{Case of cubical subsets of a cube}\label{S:TND}\hfill

\begin{defn}\label{D:subneck}
Let  $S_{a,b}$ in $\cSdp$, with $S$ a cubical subset of an $n$-cube.
The subcategory of $\slicecat{\Nec}{S_{a,b}}$ whose objects are monomorphisms $T \to S_{a,b}$ and arrows are monomorphisms between necklaces is denoted 
$\SubNeck(S_{a,b})$. The category $\SubNeck(S_{a,b})$ is actually a poset, as shown in Proposition \ref{P:subneckposet}.
\end{defn}

\begin{prop} \label{Ct_with_subneck}
Let $S_{a,b} \in \cSdp,$ with $S$ a cubical subset of an $n$-cube. The rigidification functor has the following expression:
\[\Cbox(S)(a,b)=\Cbox_t(S_{a,b})\cong \colim_{\SubNeck(S_{a,b})}\Cbox_t(T).\]
\end{prop}
\begin{proof}
We use Lemma \ref{L:TNDmono} and its notation.
Recall from Theorem \ref{T:Ct_with_neck} that we have $\Cbox_t(S_{a,b})\cong \colim_{\slicecat{\Nec}{S_{a,b}}}   \Cbox_t(T)$. Consider the 
inclusion functor $U:\SubNeck(S_{a,b}) \hookrightarrow \slicecat{\Nec}{S_{a,b}}$ and fix $f\in \slicecat{\Nec}{S_{a,b}}$. The category $f\downarrow U$ has  
$\pi(f):f\to \iota(f)$ as object, and hence is not empty. Let $g:f\rightarrow h$ be an object in $f\downarrow U$, so that $hg=f=\iota(f)\pi(f)$. The morphism $g$ in $\Nec$ 
admits the factorisation  $g=\iota(g)\pi(g)$. By the unique decomposition of $f$ as an epimorphism followed by a monomorphism, there exists an 
isomorphism $\alpha:T^f\to T'$, as illustrated by the following diagram:
  \[\begin{tikzcd}
    &T^f \ar[dr,hook,  "\iota(f)"] \ar[d, dashed, "\cong", "\alpha"'] & \\
    T\ar[ur, two heads,"\pi(f)"] \ar[dr, swap, "g"]\ar[r,two heads, "\pi(g)"] &T'\ar[d,hook, "\iota(g)"] & S_{a,b} \\ & V\ar[ur, hook,  "h"'] &
    \end{tikzcd}\]
In conclusion, the morphism $\iota(g)\alpha$ is a monomorphism, hence a morphism in $f\downarrow U$ from $\pi(f)$ to $g$.  
Thus the category $f\downarrow U$ is  connected.  
We have proved that $U$ is a final functor, from which the statement follows 
(cf. \cite[Section IX.3]{MacLane98}).
\end{proof}

\section{Quillen equivalence} \label{main-section}
In this section, we prove that the adjunction $\Cbox\dashv\Nbox$ is a Quillen equivalence between the Joyal model structure on $c\Set$ and the Bergner model structure on $s\Cat$.

\subsection{Properties of the functor $\Cbox$}

\begin{prop} \label{C_preserves_cof}
The functor $\Cbox$ preserves cofibrations.
\end{prop}

\begin{proof} Since  cofibrations of $c\Set$ are generated by $\partial^n : \partial\square^n \to \square^n$, it is enough to prove that
$\Cbox(\partial^n) : \Cbox(\partial\square^n) \to \Cbox(\square^n)$ is a cofibration.
We claim  that the  diagram
\[\begin{tikzcd}
\Sigma(\Cbox_t(\partial\square^n)) \ar[r] \ar[d, "\Sigma(\Cbox_t(\partial^n))", swap] & \Cbox(\partial\square^n) \ar[d, "\Cbox(\partial^n)"] \\
\Sigma(\Cbox_t(\square^n)) \ar[r] & \Cbox(\square^n)
\end{tikzcd}\]
where the horizontal morphisms are given  by the counit of the adjunction $\Sigma\dashv \Hom$ of Proposition \ref{adj_Sigma_Hom},  is a pushout diagram.
The set of objects of the simplicial categories on the right hand side of the diagram is in bijection with $(\square^n)_0=(\partial\square^n)_0$, which is a bounded poset. Let $a$ and $b$ be two such objects.
If $(a,b)\neq(\alpha,\omega)$, then $\square^{d(a,b)} \subseteq \partial\square^n$, so that the functor
$\slicecat{\Nec}{(\partial\square^n)_{a,b}} \to \slicecat{\Nec}{(\square^n)_{a,b}}$ is an isomorphism of categories. Theorem 
\ref{T:Ct_with_neck} implies then that the map $\Cbox(\partial^n)(a,b) : \Cbox(\partial\square^n)(a,b) \to \Cbox(\square^n)(a,b)$ is an 
isomorphism.
We conclude by Proposition \ref{P:sigmacolim}.
 We show next that $\Cbox_t(\partial\square^n) \to \Cbox_t(\square^n)$ is a cofibration. Indeed $\partial\square^n$ is a cubical subset of $
\square^n$, hence $\Cbox_t(\partial\square^n)\cong \colim_{T\in\SubNeck(\partial\square^n)}\Cbox_t(T)$  by Proposition \ref{Ct_with_subneck}, so 
that $\Cbox_t(\partial\square^n) \to \Cbox_t(\square^n)$ is a cofibration by Lemma \ref{lem_subneck}. The functor $\Sigma$ preserves 
cofibrations by Proposition \ref{adj_Sigma_Hom}, and so do pushout diagrams. This ends the proof.
\end{proof}

We refer to  Definition \ref{D:boiteinterne} for the notation in the next lemma.

\begin{lem}\label{neck_critical_edge} 
In the diagram below the horizontal arrows (where $p$ is the quotient map) are isomorphisms of simplicial sets:

\[\begin{tikzcd}
\Cbox(\sqcap^n_{i,\epsilon})(a,b) \ar[r, "\Cbox(p)","\cong"'] \ar[d] &\Cbox(\widehat{\sqcap}^n_{i,\epsilon})(pa,pb) \ar[d] \\
\Cbox(\square^n) \ar[r, "\Cbox(p)","\cong"'](a,b) &\Cbox(\widehat{\square}^n_{i,\epsilon})(pa,pb)
\end{tikzcd}\]
when $a\not=\{i\}$ if $\epsilon=1$, and $b\not=\{1,\ldots,n\}\setminus \{i\}$ if $\epsilon=0$.
\end{lem}

\begin{proof}
We prove the isomorphism $\Cbox(p):\Cbox(\square^n) \to \Cbox(\widehat{\square}^n_{i,\epsilon})$ with $\epsilon=1$ and $i=1$, so that 
the critical edge $e_{1,1}:\square^1\rightarrow \square^n$ corresponds to the edge from $\alpha=\emptyset$ to $\{1\}$. 
The other cases are similar.
Since $\Cbox$ preserves colimits, we have the pushout diagram
\[\xymatrix{\Cbox(\square^1)\ar[r]^{\Cbox(e_{1,1})}\ar[d]& \Cbox(\square^n)\ar[d]\\
\Cbox(\square^0)\ar[r] &\Cbox(\widehat{\square}^n_{1,1})}\]
Let us define the following simplicial category $S$. The set of objects of $S$ is identified with that of $\widehat{\square}^n_{1,1}$. We denote by 
$\bar\alpha=p(\alpha)=p(\{1\})$. 
Define $S(\bar\alpha,\bar\alpha)=*$, $S(\bar\alpha,b)=\Cbox(\square^n)(\alpha,b)$ for $b\not=\bar\alpha$ and $S(a,b)=\Cbox(\square^n)(a,b)$, for 
$a\not=\bar\alpha$. The composition is induced by that of $\Cbox(\square^n)$. 
Let $\pi:\Cbox(\square^n)\rightarrow S$ be the map which coincides with $p$ on objects, and  is the identity on morphisms except for the case 
$\pi:\Cbox(\square^n)(\{1\},b)\rightarrow S(\bar\alpha,b)$, for which we use the composite
\[\xymatrix{
\Cbox(\square^n)(\{1\},b)\ar[r]&\Cbox(\square^n)(\{1\},b)\times \Cbox(\square^n)(\alpha,\{1\})\ar[r]^-{\circ}&
\Cbox(\square^n)(\alpha,b)\cong S(\bar\alpha,b)},\]
which is well defined since $\Cbox(\square^n)(\alpha,\{1\})=*$. 
One checks  easily that $S$ satisfies the universal pushout property, hence $\Cbox(\widehat{\square}^n_{1,1}) \cong  S$. We conclude, 
since by definition of $S$ we have
\[\Cbox(\widehat{\square}^n_{1,1})(pa,pb)\cong S(pa,pb) = \Cbox(\square^n)(a,b)\quad\mbox{if}\; a\not=\{1\}.
\]
The side condition is needed since, for $a=\{1\}$ and $b\succ 1$, we have $S(p(\{1\}),b)=S(\bar\alpha,b)= \Cbox(\square^n)(\alpha,b)\not\cong
\Cbox(\square^n)(\{1\},b)$.
The proof for the inner box is exactly the same since $\Cbox(\sqcap^n_{1,1})(\alpha,\{1\})=*$.
\end{proof}

\begin{prop}\label{C_hn}
$\Cbox(h^n_{i,\epsilon}) : \Cbox(\widehat{\sqcap}^n_{i,\epsilon}) \to \Cbox(\widehat{\square}^n_{i,\epsilon})$ is an acyclic cofibration.
\end{prop}

\begin{proof}
The proof is analogous to the proof of Proposition  \ref{C_preserves_cof}. Note that the set of vertices of 
$\widehat{\sqcap}^n_{i,\epsilon}$  coincides  with that of  $\widehat{\square}^n_{i,\epsilon}$ and is a bounded poset. 
By Proposition \ref{P:sigmacolim},
in order to establish that the diagram
\[\begin{tikzcd}
\Sigma(\Cbox_t(\widehat{\sqcap}^n_{i,\epsilon})) \ar[r] \ar[d, swap, "\Sigma(\Cbox_t(h^n_{i,\epsilon}))"] & 
\Cbox(\widehat{\sqcap}^n_{i,\epsilon}) \ar[d, "\Cbox(h^n_{i,\epsilon})"] \\
\Sigma(\Cbox_t(\widehat{\square}^n_{i,\epsilon})) \ar[r] & \Cbox(\widehat{\square}^n_{i,\epsilon})
\end{tikzcd}\]
is a pushout diagram, we have to show that the maps $\Cbox(h^n_{i,\epsilon})(pa,pb) : \Cbox(\widehat{\sqcap}^n_{i,\epsilon})(pa,pb) \to 
\Cbox(\widehat{\square}^n_{i,\epsilon})
(pa,pb)$ are isomorphisms, for every $(a,b)$ such that $(pa,pb)\not=(p\alpha,p\omega)$, or equivalently using Lemma \ref{neck_critical_edge}, that 
$\Cbox(h^n_{i,\epsilon})(a,b) : 
\Cbox({\sqcap}^n_{i,\epsilon})(a,b) \to \Cbox({\square}^n_{i,\epsilon})(a,b)$ is an isomorphism. The latter is established exactly as in the proof 
of Proposition  \ref{C_preserves_cof}, noticing that $\square^{d(a,b)} \subseteq \sqcap^n_{i,\epsilon}$ for such $(a,b)$.
What remains to prove is that {$\Cbox_t(\widehat{\sqcap}^n_{i,\epsilon})\rightarrow \Cbox_t(\widehat{\square}^n_{i,\epsilon})$ is an acyclic 
cofibration, or equivalently, by Lemma \ref{neck_critical_edge} again, that}
$\Cbox_t({\sqcap}^n_{i,\epsilon})\to \Cbox_t({\square}^n_{i,\epsilon})$ is an acyclic cofibration for the Kan-Quillen structure on simplicial sets. We 
already know that it is a  cofibration, since  $\Cbox$ preserves cofibrations by Proposition \ref{C_preserves_cof} and so does $\Cbox_t$.
Since  $\Cbox_t(\square^n)$ is contractible (cf. Lemma \ref{Ct_cubes_contractible}), showing that 
$\Cbox_t(\sqcap^n_{i,\epsilon}) \to \Cbox_t(\square^n)$  is acyclic amounts to proving that $\Cbox_t(\sqcap^n_{i,\epsilon})$ is contractible.
 From Proposition \ref{Ct_with_subneck}, we have
 \[\Cbox_t(\sqcap^n_{i,\epsilon})=\colim\limits_{T\in \SubNeck(\sqcap^n_{i,\epsilon})}\Cbox_t(T).\] 
 Furthermore, it follows from Section \ref{S:B1} that the category $\SubNeck(\sqcap^n_{i,\epsilon})$ is direct (and hence is Reedy). The
 diagram $(T\to\sqcap^n_{i,\epsilon})\mapsto \Cbox_t(T)$ is Reedy cofibrant: for every $T\in \SubNeck(\sqcap^n_{i,\epsilon})$, 
 the latching morphism 
 \[\colim\limits_{\substack{U \in \SubNeck(T) \\ U\neq T}}\Cbox_t(U) \to \Cbox_t(T)\] 
 is a monomorphism by Lemma \ref{lem_subneck}.
It follows from \cite[Theorem 19.9.1]{H}, and from the fact that any direct category has fibrant constants, that
  the natural map
 \[\hocolim_{T\in \SubNeck(\sqcap^n_{i,\epsilon})} \Cbox_t(T)\rightarrow \colim_{T\in \SubNeck(\sqcap^n_{i,\epsilon})}\Cbox_t(T)\]
 is a weak equivalence of simplicial sets. In conclusion we have
\[
\Cbox_t(\sqcap^n_{i,\epsilon})   \sim \hocolim\limits_{T\in \SubNeck(\sqcap^n_{i,\epsilon})}\Cbox_t(T)\\
 \sim 
\hocolim\limits_{T\in \SubNeck(\sqcap^n_{i,\epsilon})}* \sim  N(\SubNeck(\sqcap^n_{i,\epsilon})) \sim  *
\]
since $\Cbox_t(T)$ and $N(\SubNeck(\sqcap^n_{i,\epsilon}))$ are contractible, by Corollary \ref{Ct_nacklaces_contractible} 
and Proposition \ref{C:nerf_subneck}.
\end{proof}

\subsection{Quillen adjunction}

We shall use the following result of Joyal \cite[E.2.14]{J}.

\begin{prop}\label{Joyal_result}
An adjunction $L\dashv R$ between two model categories is Quillen if and only if
$L$ preserves cofibrations and 
$R$ preserves fibrations between fibrant objects.
\end{prop}

In view of this proposition, what remains to prove is that $\Nbox$ sends fibrations between fibrant categories to fibrations between cubical quasi-categories. 
We shall use two lemmas, which we now present. We refer to Definition \ref{D:special} for the definition of equivalence and that of special open box. The first lemma is tautological.

\begin{lem} \label{lem_fib_technical2}
Let $\EuScript{A}$ be a fibrant simplicial category, $v : \Cbox(\square^1) \to \EuScript{A}$ and $\tilde{v} : \square^1\to\Nbox(\EuScript{A})$ 
its transpose. As $\Cbox(\square^1)$ is the simplicial category with only one non-trivial arrow, we  see $v$ as an arrow in $\EuScript{A}$.
Then $\tilde{v}$ is an equivalence in the cubical quasi-category $\Nbox(\EuScript{A})$ if and only if $\pi_0(v)$ is an isomorphism in $\pi_0(\EuScript{A})$.
\end{lem}

\begin{lem} \label{lem_fib_technical1}
Let $p:X \to Y$ be an inner fibration between cubical quasi-categories and $\sqcap^2_{i,\epsilon} \to X$ a special open box. Any commutative square of the following form has a lift:
\[\begin{tikzcd}
\sqcap^2_{i,\epsilon}  \ar[r] \ar[d] & X \ar[d, "p"] \\
\square^2 \ar[ur, dashed] \ar[r] & Y
\end{tikzcd}\]
\end{lem}

\begin{proof} It is a special case of \cite[Lemma 4.14]{DKLS}. We give a proof for the case $(i,\epsilon) = (1,1)$, the other cases being similar.
We represent the map $\sqcap^2_{1,1} \to X$ on the left below, with $f$ an equivalence. 
 The $1$-cube $f$ being an equivalence, the middle diagram below exists.
 Glueing the two diagrams,
we get the following partially filled 3-cube in $X$ (on the right):
\[
\begin{tikzcd}
\ar[r, "f"] \ar[d, "u"] & {}\\
\ar[r, "v"] & {}
\end{tikzcd}
\quad\quad
\begin{tikzcd}
\ar[r, equal] \ar[d, "f"]& \ar[d, equal] \\
 \ar[r,  "g"] &  {}
\end{tikzcd}
\quad\quad
\begin{tikzcd}[column sep=15pt, row sep=10pt]
& \ar[rr, equal]  {} & & {} \\
\ar[ur, equal] \ar[rr, "\ \ \ \ f"] \ar[dd, swap, "u"] & & \ar[ur, "g" near start]  {} & \\
& & & {} \\
\ar[rr, "v"]  & & {}
\end{tikzcd}
\]
where our conventions for the coordinates in dimensions 2 and 3 are as follows:
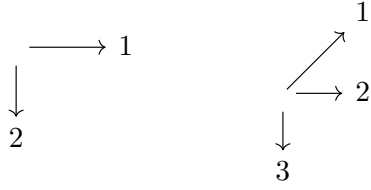
\[\begin{tikzcd}
\ar[r] \ar[d] & 1\\
2 & {}
\end{tikzcd}
\quad\quad\quad\quad
\begin{tikzcd}[column sep=15pt, row sep=15pt]
& 1\\
\ar[ur]
\ar[r] \ar[d] & 2\\
3 & {}
\end{tikzcd}
\]
The map $B:\square^2\rightarrow Y$ is represented by the following  $2$-cube in $Y$:\qquad
\begin{tikzcd}
\ar[r, "pf"] \ar[d, "pu"]& \ar[d, "w"] \\
 \ar[r,  "pv"] &  {}
\end{tikzcd}

The proof goes in 3 steps. For the first step, we assume $Y=*$.
We complete the above 3-cube cube progressively, as follows:

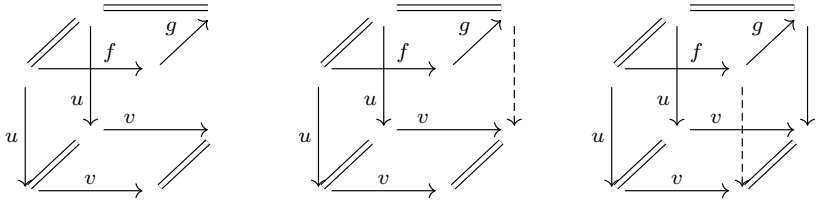
\[
\begin{tikzcd}[column sep=15pt, row sep=15pt]
& \ar[rr, equal] \ar[dd, swap, "u" near end] & & {} \\
\ar[ur, equal] \ar[rr, "\ \ \ \ f"] \ar[dd, swap, "u"] & & \ar[ur, "g"] & \\
& \ar[rr, "v" near start] & & {} \\
\ar[rr, "v"] \ar[ur, equal] & & \ar[ur, equal]
\end{tikzcd}\qquad
\begin{tikzcd}[column sep=15pt, row sep=15pt]
& \ar[rr, equal] \ar[dd, swap, "u" near end] & & \ar[dd, dashed] \\
\ar[ur, equal] \ar[rr, "\ \ \ \ f"] \ar[dd, swap, "u"] & & \ar[ur, "g"]  & \\
& \ar[rr, "v" near start] & & {} \\
\ar[rr, "v"] \ar[ur, equal] & & \ar[ur, equal]
\end{tikzcd}\qquad
\begin{tikzcd}[column sep=15pt, row sep=15pt]
& \ar[rr, equal] \ar[dd, swap, "u" near end] & & \ar[dd] \\
\ar[ur, equal] \ar[rr, "\ \ \ \ f"] \ar[dd, swap, "u"] & & \ar[ur, "g"] \ar[dd, dashed] & \\
& \ar[rr, "v" near start] & & {} \\
\ar[rr, "v"] \ar[ur, equal] & & \ar[ur, equal]
\end{tikzcd}\]

The top face is full by hypothesis, the left and the bottom faces are given by degeneracies. 
Because $X$ is a cubical quasi-category, 
the back face (picture in the middle), then the right face (picture on the right), and finally the whole cube, and hence a fortiori the front face, can be filled.
In consequence any special open box in a cubical quasi-category $X$ can be filled by a $2$-cube in $X$.

The second step consists in filling the following $3$-cube in $Y$, where the front, top,  left and bottom faces are already filled.

\[
\begin{tikzcd}[column sep=15pt, row sep=15pt]
& \ar[rr, equal] \ar[dd, swap, "pu" near end] & & \ar[dd,dashed,"\rho"] \\
\ar[ur, equal] \ar[rr, "\ \ \ \ pf"] \ar[dd, swap, "pu"] & & \ar[ur, "pg"] \ar[dd,"w"]& \\
& \ar[rr, "pv" near start] & & {} \\
\ar[rr, "pv"] \ar[ur, equal] && \ar[ur, equal]
\end{tikzcd}\]
By \cite[Lemma 2.6]{DKLS}, $g$ is an equivalence because $f$ is and so is $pg$. Hence the right face is a special open box in the cubical 
quasi-category $Y$ and thus can be filled by step 1. Then the whole cube is filled because the critical edge associated to the back face is the identity.

For the last step, we resume the filling of the $3$-cube of the first step (with the same pictures as above) in $X$, but now in the general case. 
The aim is to fill in the front face of the $3$-cube in $X$  by a $2$-cube $A$ satisfying $pA=B$. 
Because $p$ is an inner fibration, its back face can be filled by a $2$-cube such that the  dashed arrow in the picture in the middle is sent to $\rho$ by $p$. 
Then the same is true for its right face, so that  the  dashed arrow in the picture on the right is sent to $w$ by $p$. Finally, the whole cube is filled and sent by 
$p$ to the $3$-cube in $Y$, and a fortiori its front face $A$ satisfies $pA=B$. \end{proof}

\begin{prop}\label{N_preserves_fib}
The functor $\Nbox$ sends fibrations between fibrant simplicial categories to fibrations between cubical quasi-categories.
\end{prop}
\begin{proof}
Using Proposition \ref{C_hn}, we conclude by adjunction
 that  $\Nbox(\EuScript{C})$ is a cubical quasi-category if $\EuScript{C}$ is a fibrant simplicial category, and  that if $f : \EuScript{C} \to \EuScript{D}$ 
 is a DK-fibration between fibrant simplicial categories, then $\Nbox(f)$ is an inner fibration between cubical quasi-categories. By Theorem 
 \ref{Cset-Joyal}, we are left to show that $\Nbox(f)$ has the right lifting property with respect to the endpoint inclusions $j_0:\{ 0 \} {\to} K$ and $j_1:\{ 1 \} {\to} K$. 
These cases being similar, we only treat the first one. Consider a commutative square:
\[\begin{tikzcd}
\{0\} \ar[r, "\bar{a}"] \ar[d, "j_0"] & \Nbox\EuScript{C} \ar[d, "\Nbox f"] \\
K \ar[r] & \Nbox\EuScript{D}
\end{tikzcd}\]

We shall first lift the middle vertical edge of $K$. By Lemma \ref{lem_fib_technical2}, its image in $\Nbox(\EuScript{D})$ 
corresponds to some arrow $v\in\EuScript{D}_0(a,b)$, which is an isomorphism in $\pi_0(\EuScript{D})$.  The same will have to be true for its 
image in $\Nbox(\EuScript{C})$ through the lifting.
The object $\bar{a}$ of $\EuScript{C}$  satisfies $f(\bar a)=a$.  We proceed as folllows.
        \begin{itemize}
            \item Since $\pi_0(f)$ is an isofibration of categories, there exists some $\bar{b}\in \Ob(\EuScript{C})$ and some 
            $v'\in\EuScript{C}_0(\bar{a},\bar{b})$ such that $f(\bar{b})=b$, $\pi_0(f(v')) = \pi_0(v)$, and $\pi_0(v')$ is an isomorphism in $\pi_0(\EuScript{C})$.
            \item Since $\EuScript{D}(a,b)$ is a Kan complex, we can find a 1-simplex $\delta\in\EuScript{D}_1(a,b)$ such that $\partial_1\delta = v$ and $\partial_0\delta = f(v')$.
            \item Since $f_{\bar{a},\bar{b}} : \EuScript{C}(\bar{a},\bar{b}) \to \EuScript{D}(a,b)$ is a Kan fibration, we can lift 
            $\delta\in\EuScript{D}_1(a,b)$ to some $\bar{\delta}\in\EuScript{C}_1(\bar{a},\bar{b})$ satisfying $\partial_0\bar{\delta} = v'$.
                  \end{itemize}
Then $\bar{v} = \partial_1\bar{\delta}\in\EuScript{C}_0(\bar{a},\bar{b})$ meets our goal, i.e., 
satisfies $f(\bar{v}) = v$, and $\pi_0(\bar{v}) = \pi_0(v')$ is an isomorphism in $\pi_0(\EuScript{C})$, so that,
                   $\bar{v}$ seen as an edge in $\Nbox(\EuScript{A})$ is an equivalence, by Lemma \ref{lem_fib_technical2}. Therefore, the two open boxes in
                  \[\xymatrix{1\ar@{=}[d]&0\ar[d]_{\bar{v}}\ar@{=}[r]&0\ar@{=}[d]\\
1\ar@{=}[r]&1&0}
\]
are special. Calling this diagram $\bar{v}$, our lifting problem reduces now to the following one:
\[\begin{tikzcd}
K' \ar[r, "\bar{v}"] \ar[d] & \Nbox\EuScript{C} \ar[d, "\Nbox f"] \\
K \ar[r] & \Nbox\EuScript{D}
\end{tikzcd}\]
where $K'$ is $K$ without its $2$-cubes and without its horizontal nondegenerate $1$-cubes.
This is performed by applying Lemma \ref{lem_fib_technical1} to each of the two special boxes above.
\end{proof}

\begin{prop}
The adjunction $\Cbox \dashv \Nbox$ is Quillen.
\end{prop}
\begin{proof} This follows from Propositions \ref{C_preserves_cof} and  \ref{N_preserves_fib}, thanks to Joyal's characterisation recalled in Proposition \ref{Joyal_result}.
\end{proof}

\subsection{Quillen equivalence}
In order to prove that the Quillen adjunction $\Cbox \dashv \Nbox$ is a Quillen equivalence, we first compare it with the simplicial rigidification 
$ \C^{\Delta}$ using the functor $Q$ of Section \ref{S:Q} and then use the Quillen equivalences induced by $Q$ and $ \C^{\Delta}$.

\begin{lem}\label{CQn_nat_Deltan}
There exists a morphism  $\phi_n:\Cbox(Q^n) \to \C^{\Delta}(\Delta^n)$ in $s\Cat$, which is a bijection on objects and  is natural in $[n]\in\Delta$, \ie a natural 
transformation $\phi:\Cbox\circ Q\Rightarrow \C^{\Delta}\circ Y$.
\end{lem}

\begin{proof}
We start by defining a family of morphisms $\psi_n:\Cbox(\square^n) \to \C^{\Delta}(\Delta^n)$ in $s\Cat$. 
On objects,  we set $\psi_n(a)=\sup a$ (cf. Lemma \ref{L:vertofQ}).
If $a\preccurlyeq b$ ($a\subseteq b$), then $\sup a\leq \sup b$, and we define a map  $\psi_n(a,b):\Cbox(\square^n)(a,b)\to \C^{\Delta}(\Delta^n)(\sup a,\sup b)$ in $s\Set$, 
as follows. Since $\Cbox(\square^n)(a,b)$ is the nerve of the poset $\Sigma_{b\setminus a}$ with the weak  order $\leadsto_B$ (see Section \ref{S:pathcatneck}), and  $\C^{\Delta}(\Delta^n)(\sup a,\sup b)$ is the 
nerve of  the poset $\mathcal P(]\sup a,\sup b[)$ with the inclusion order, we define this map at the level of the underlying posets.
Let $k=d(a,b)$ and $(x_1,\ldots,x_k)\in \Sigma_{b\setminus a}$. We set
\[
\tilde\psi_n(a,b)(x_1,\ldots,x_k)
=\{x_l\:| \;\forall p<l, x_p<x_l\} \; \cap \;]\sup a,\sup b[.\]

The map $\tilde\psi_n(a,b)$ is a morphism of posets. Assume  $x_r>x_{r+1}$ for some $r$. Then $x:=(x_1,\ldots,x_k)\leadsto_B  (x_1,\ldots,x_{r+1},x_r,\ldots x_k)=:y$.
We write $A(x)=\{x_l\:| \;\forall p<l, x_p<x_l\}$. Then we observe   that  $A(x)\setminus\{x_r,x_{r+1}\}=A(y)\setminus\{x_r,x_{r+1}\}$, that 
$x_{r+1}\not\in A(x)$, and that $(x_r\in A(x))\Rightarrow (x_r\in A(y))$. It follows that $A(x)\subseteq A(y)$, and hence $\tilde\psi_n(a,b)(x)\subseteq \tilde\psi_n(a,b)(y)$.

We next show that  $\tilde\psi_n$ preserves the concatenation product. Assume $a\preccurlyeq b\preccurlyeq c$. Let $x:=(x_1,\ldots,x_k)\in \Sigma_{b\setminus a}$, 
$y:=(y_1,\ldots, y_{k'})\in \Sigma_{c\setminus b}$. We set
$z:=(x_1,\ldots,x_k,y_1,\ldots,y_{k'})$.
We have to prove  that if $\sup b\not\in\{\sup a,\sup c\}$ (equivalently $\sup a<\sup b<\sup c$), then 
\[\tilde\psi_n(a,c)(z)=\tilde\psi_n(a,b)(x)\cup\{\sup b\}\cup \tilde\psi_n(b,c)(y).\]
We observe that $A(z)$  splits as $A(x)\cup B$, where $B\subseteq A(y)$ and $A(y) \,\cap\,]\sup b,\sup c]\subseteq B$. 
This settles the left-to-right inclusion, as well as the inclusions $\tilde\psi_n(a,b)(x)\subseteq \tilde\psi_n(a,c)(z)$ and  $\tilde\psi_n(b,c)(y)\subseteq \tilde\psi_n(a,c)(z)$.
Since $\sup b>\sup a$, we have $\sup(b\setminus a)=\sup b$. Thus there exists $l$ such that $x_l=\sup b$ and $(\forall p\in\{1,\ldots,k\}\; x_p\leq \sup b$), and  a fortiori
$\{\sup b\}\subseteq \tilde\psi_n(a,c)(z)$ holds.

In conclusion, setting $\psi_n=N(\tilde\psi_n)$, we have shown that
$\psi_n:\Cbox(\square^n) \to \C^{\Delta}(\Delta^n)$
is an enriched functor of simplicial categories. 

Let us prove that $\psi_n$ factors through the map $\Cbox(\pi_n):\Cbox(\square^n) \to \Cbox(Q^n)$, where $\pi_n$ is the quotient map of Definition \ref{D:Q}.  
As $\Cbox$ is cocontinuous, from Definition \ref{D:Q}, the following diagram is a pushout: 
\[\begin{tikzcd}[column sep=105pt]
\Cbox(\square^0\otimes\square^{n-1}) \bigsqcup \Cbox(\square^1\otimes\square^{n-2}) \bigsqcup \dots \bigsqcup \Cbox(\square^{n-1}\otimes\square^0) \arrow{d} 
\arrow{r}{(\Cbox(\partial_{1,1}), \Cbox(\partial_{2,1}), \dots, \Cbox(\partial_{n,1}))} & \Cbox(\square^n) \arrow[d,  "\Cbox(\pi_n)"''] \\
\Cbox(\square^{n-1}) \bigsqcup \Cbox(\square^{n-2}) \bigsqcup \dots \bigsqcup \Cbox(\square^0) \arrow{r} & \Cbox(Q^n)
\end{tikzcd}\]
So, by universality, all we need to get our factorisation is a commutative square
\[\begin{tikzcd}[column sep=110pt]
\Cbox(\square^0\otimes\square^{n-1}) \bigsqcup \Cbox(\square^1\otimes\square^{n-2}) \bigsqcup \dots \bigsqcup \Cbox(\square^{n-1}\otimes\square^0) \arrow{d} \arrow[dr, phantom]
\arrow{r}{(\Cbox(\partial_{1,1}), \Cbox(\partial_{2,1}), \dots, \Cbox(\partial_{n,1}))} & \Cbox(\square^n)\arrow[d, "\psi_n"] \\
\Cbox(\square^{n-1}) \bigsqcup \Cbox(\square^{n-2}) \bigsqcup \dots \bigsqcup \Cbox(\square^0) \arrow[r, dashed] & \C^{\Delta}(\Delta^n)
\end{tikzcd}\]
i.e., for all $1\leq i \leq n$, we want a lift  in the following diagram:
\[\begin{tikzcd}
\Cbox(\square^{i-1}\otimes\square^{n-i}) \arrow{r}{\Cbox(\partial_{i,1})} \arrow{d} & \Cbox(\square^n) \arrow{r}{\psi_n} & \C^{\Delta}(\Delta^n) \\
\Cbox(\square^{n-i}) \arrow[urr, "\gamma_{i,n}", dashed, swap]
\end{tikzcd}\]

We define first $\tilde\gamma_{i,n}$ as a map between the underlying posets.
Let $a\otimes a'$ be a vertex of $\square^{i-1}\otimes\square^{n-i}$.
We note that $(\psi_n\circ\partial_{i,1})(a\otimes a')=\sup(a\cup\{i\}\cup a')$ is independent of $a$ and we set $\tilde\gamma_{i,n}(a')=\sup(\{i\}\cup a')$.
Let $a\otimes a'$ and $b\otimes b'$ be two objects of $\square^{i-1}\otimes\square^{n-i}$,  with $a\preccurlyeq b$ and $a'\preccurlyeq b'$, and
 $(x_1,\ldots,x_k;y_1,\ldots,y_{k'})$  be an element of $\Sigma_{b\setminus a}\times\Sigma_{b'\setminus a'}$. 
 We have
 \begin{multline*}
 (\psi_n\circ\partial_{i,1})(a\otimes a';b\otimes b')(x_1,\ldots,x_k;y_1,\ldots, y_{k'})=\\
 \{y_l | \;\forall p<l, y_p<y_l\} \; \cap \;]\sup(\{i\}\cup a'),\sup(\{i\}\cup b')[,
 \end{multline*}
 so that $\tilde\gamma_{i,n}$ is a well defined morphism of posets, and hence  $\gamma_{i,n}=N(\tilde\gamma_{i,n})$ provides the required lifting.

As a consequence, there is a well defined morphism $\phi_n: \Cbox(Q^n)\to\C^{\Delta}(\Delta^n)$ in $s\Cat$ for each $[n]$, which is a bijection on objects. 
It remains to show that it yields a natural transformation $\phi:\Cbox\circ Q\Rightarrow \C^{\Delta}\circ Y$. Namely, given $u:[n]\rightarrow [m]$ in $\Delta$ and denoting 
the induced map by $u_*:Q^n\rightarrow Q^m$, we have to prove the commutativity of the diagram:
\[\begin{tikzcd}
\Cbox(Q^{n}) \arrow[d,"\phi_{n}", swap] \arrow{rr}{\Cbox(u_*)}& & \Cbox(Q^m)  \arrow{d}{\phi_m} \\
\C^{\Delta}(\Delta^{n}) \arrow[rr,"\C^{\Delta}(u)"] && \C^{\Delta}(\Delta^m)
\end{tikzcd}\]
Lemma \ref{L:vertofQ} implies that it is commutative at the level of objects.
It is enough to check it  for $u$  a face $d_j$ or a degeneracy $s_j$
(left to the reader).
\end{proof}

\iffalse
\[\begin{tikzcd}
\Cbox(\square^{n-1}) \arrow{d}{\psi_{n-1}} \arrow{r}{\Cbox(\partial_{1,1})} & \Cbox(\square^{n})  \arrow{d}{\psi_n} \\
\C^{\Delta}(\Delta^{n-1}) \arrow{r}{\C^{\Delta}(d_0)} & \C^{\Delta}(\Delta^n)
\end{tikzcd}
\begin{tikzcd}
\Cbox(\square^{n-1}) \arrow{d}{\psi_{n-1}} \arrow{r}{\Cbox(\partial_{i,0})} & \Cbox(\square^{n}) \arrow{d}{\psi_n} \\
\C^{\Delta}(\Delta^{n-1}) \arrow{r}{\C^{\Delta}(d_i)} & \C^{\Delta}(\Delta^n)
\end{tikzcd}\]
\[\begin{tikzcd}
\Cbox(\square^{n+1}) \arrow{d}{\psi_{n+1}} \arrow{r}{\Cbox(\sigma_1)} & \Cbox(\square^{n}) \arrow{d}{\psi_n} \\
\C^{\Delta}(\Delta^{n+1}) \arrow{r}{\C^{\Delta}(s_0)} & \C^{\Delta}(\Delta^n)
\end{tikzcd}
\begin{tikzcd}
\Cbox(\square^{n+1}) \arrow{d}{\psi_{n+1}} \arrow{r}{\Cbox(\gamma_{i,0})} & \Cbox(\square^{n}) \arrow{d}{\psi_n} \\
\C^{\Delta}(\Delta^{n+1}) \arrow{r}{\C^{\Delta}(s_i)} & \C^{\Delta}(\Delta^n)
\end{tikzcd}\]
\fi

\begin{rmk}
Note that $\C^{\Delta}\circ Y \ncong \Cbox \circ Q$. 
\end{rmk}

\begin{prop}\label{comparaison_simpliciale}
The natural transformation of Lemma \ref{CQn_nat_Deltan} induces a natural DK-equivalence  $\phi: \Cbox \circ Q \Longrightarrow \C^{\Delta}$.
\end{prop}

\begin{proof}
We follow closely the proof of \cite[Proposition 6.21]{DKLS}.
Every simplicial set $S$ is a colimit of representables, and the functors $\Cbox\circ Q$ and $\C^{\Delta}$ are left adjoint, hence 
preserve colimits.  It follows that we can upgrade the  natural transformation of Lemma \ref{CQn_nat_Deltan} as $\phi:\Cbox \circ Q \Rightarrow 
\C^{\Delta}$, which is componentwise a bijection on objects by Lemma \ref{L:cocontinuous}. 
We prove first that if $S$ is $k$-skeletal for 
some $k$, then $\phi_S$ is a DK-equivalence. If $k=0$ and $k=1$, this is an isomorphism of simplicial categories. Assume it is true for every 
$k<n$. The functors $\Cbox\circ Q$ and $\C^\Delta$ preserve cofibrations, as well as Joyal weak equivalences (since every object of $s\Set$ is 
cofibrant). Given any $1\leq i\leq n-1$, the cofibration $\Lambda_i^n\rightarrow \Delta^n$ is a Joyal weak equivalence of simplicial sets and $
\Lambda_i^n$ is $(n-1)$-skeletal, so that $\phi_{\Delta^n}$ is a DK-equivalence in $s\Cat$, by the 2 out of 3 property. Given  an 
$(n-1)$-skeletal simplicial set $X$, let us consider the pushout diagram 
$\mathcal X$:

\[\xymatrix{\partial\Delta^n\ar[r]\ar[d] & X\ar[d]\\
\Delta^n\ar[r]& X'}\]
The  left vertical arrow is a cofibration. Applying our functors, we get a cube diagram, where the front face is $(\Cbox\circ Q)(\mathcal X)$, the back 
face is $\C^{\Delta}(\mathcal X)$. Both are pushout diagrams, and their left vertical arrow is a cofibration.  By induction hypothesis and by the proof 
above, the morphisms $\phi_{\partial\Delta^n}$, $\phi_{\Delta^n}$ and $\phi_X$ are DK-equivalences. We can thus
apply \cite[Proposition 15.10.10]{H} and conclude that $\phi_{X'}$ is a DK-equivalence. This completes the proof of our claim (making repeated use 
of such pushouts).
Finally, we observe that any simplicial set $S$ is a sequential colimit of cofibrations (the family of  inclusions of the $n$-
skeleton into the $(n+1)$-skeleton), preserved by the two functors and thus entailing that $\phi_S$ is a DK-equivalence, by \cite[Proposition 
15.10.12]{H}. 
\end{proof}

\begin{cor}
The adjunction $\Cbox \dashv \Nbox$ is a Quillen equivalence.
\end{cor}
\begin{proof}
The proposition above implies that the total left derived functor $\mathbb{L}(\Cbox\circ Q)$ is isomorphic to $\mathbb{L}(\C^{\Delta})$. But $\C^{\Delta}$ is a 
Quillen equivalence (Theorem \ref{sSet-sCat-quillen-equivalence}), hence $\Cbox\circ Q$ is also a Quillen equivalence. We conclude  by Theorem \ref{Q_Quillen} 
and  the 2 out of 3 property for Quillen equivalences.
\end{proof}

\appendix

\section{Tools in category theory}  \label{category-appendix}

In this section, we collect some categorical and enriched categorical  tools  that are needed in the paper.

\subsection{Wedge sum and concatenation}\label{S:wedge}

Let $\EuScript C$ be a category with finite colimits and terminal object $*$. Let $X$ be an object of $\EuScript C$. A point $x$ in $X$ is a map 
$x:*\rightarrow X$. We say also that $X$ is pointed by $x$.
If $x$ is a point in $X$ and $y$ is a point in $Y$, we write $X\vee Y$ for the pushout of the diagram
\[\xymatrix{X& {*}\ar[r]^-{y}\ar[l]_-{x} & Y}.\]
We will consider a similar construction in  $\EuScript C_{*,*}=\slicecat{*\sqcup *}{\EuScript C}$, the category of double pointed objects in $\EuScript C$. We will 
denote an object  $(a,b):*\sqcup *\rightarrow X$ in this category by $X_{a,b}$.

\begin{defn}\label{D:concat} Let $X_{a,b}$ and $Y_{u,v}$ be two objects of $\EuScript C_{*,*}$.
\begin{itemize}
\item  the wedge sum $X\vee Y$ of the pointed sets 
$b:*\rightarrow X$ and $u:*\rightarrow Y$ is naturally  double pointed by $(a,v):*\sqcup *\rightarrow X\vee Y$,
\item  for  $f: X_{a,b}\rightarrow X'_{a',b'}$ 
and $g:Y_{u,v}\rightarrow Y'_{u',v'}$, we denote by $f\vee g: (X\vee Y)_{a,v}\rightarrow (X'\vee Y')_{a',v'}$ the double pointed map induced by 
the universal property of the pushout  and the natural maps $X'\rightarrow X'\vee Y'$ and $Y'\rightarrow X'\vee Y'$. It endows 
$\EuScript C_{*,*}$ with a monoidal  structure, with unit $*$ (doubled pointed by itself). We call this product the {\sl concatenation product}.
\end{itemize}
Let $X$ be an object in $\EuScript C$ and let  $u,v,w$ be points in $X$. 
For any maps $f:S_{a,b}\rightarrow X_{u,v}$ and $g:T_{a',b'}
\rightarrow X_{v,w}$  in $\EuScript C_{*,*}$, we write  $f*g: (S\vee T)_{a,b'}\rightarrow X_{u,w}$ for the corresponding structure map out of 
the pushout.
\end{defn}

\subsection{$P$-shaped categories}\label{S:Pshapedcat}
We introduce the notion of $P$-shaped category, for $P$ a poset.

\begin{defn}\label{bounded-poset} A {\it bounded poset} is a poset $P$  having a least and greatest element denoted respectively by $\alpha_P$ and $\omega_P$.
\end{defn}

\begin{rmk}\label{R:veeof posets} Given two bounded posets $P,Q$,  the poset  $P\vee Q$, where $\omega_P$ and $\alpha_Q$ (see Definition \ref{D:concat}) 
are identified,  is also a bounded poset, bounded by $\alpha_P$ and $\omega_Q$. 
Note that $P\rightarrow P\vee Q$ is an embedding of posets.
\end{rmk}

In this section we fix  $(\mathcal V,\otimes,I)$ a symmetric monoidal category, with initial object denoted by $\emptyset$.

\begin{defn}  Let $P$ be a 
poset. A $\mathcal V$-enriched category $\EuScript C$ is {\sl $P$-shaped} if the following conditions are fulfilled:
\begin{itemize}
\item the set of objects of $\EuScript C$ is in bijection with  $P$, 
\item $ \forall p\in P,\; \EuScript C(p,p)=I$,
\item $\forall p,q\in P,\; \EuScript C(p,q)\not=\emptyset\Rightarrow p\leq q$.
\end{itemize}
\end{defn}

\begin{ex} Any poset $P$ gives rise to a $\mathcal V$-enriched category $\hat P$: the objects are the elements of $P$, and for every $p,q$ in $P$ we have
$\hat P(p,q)=I$, if $p\leq q$ and $\hat P(p,q)=\emptyset$, otherwise. Hence $\hat P$ is $P$-shaped.
\end{ex}

\begin{prop}\label{P:V-concat} Let $P$ and $Q$ be two bounded posets. Let $\EuScript C$ be a $P$-shaped $\mathcal V$-category and 
$\EuScript D$ be a $Q$-shaped  $\mathcal V$-category. The following data yield a 
$\mathcal V$-enriched $(P\!\vee\! Q)$-shaped category
$\EuScript E=\EuScript C\vee\EuScript D$, called the concatenation category of $\EuScript C$ and $\EuScript D$:
\begin{itemize}
\item the set  of objects of $\EuScript E$ is in bijection with $P\vee Q$, so that we can identify objects of $\EuScript E$ with elements of $P\vee Q$,
\item $\EuScript E(x,y)=\begin{cases} \EuScript C(x,y),&\text{ if } x,y\in P,\\
\EuScript D(x,y),& \text{ if } x,y\in Q,\\
\EuScript C(x,\omega_P)\otimes \EuScript D(\alpha_Q,y),&\text{ if } x\in P\text{ and } y\in Q\\
\emptyset, & else,\end{cases}$
\end{itemize}
and the composition is the obvious one. 
\end{prop}

\begin{proof} Note that $\EuScript C(\omega_P,\omega_P)=I=\EuScript D(\alpha_Q,\alpha_Q)$ implies that the definition above is consistent.
Note also that $\EuScript E$ is $P\vee Q$-shaped. We prove that $\EuScript E$ satisfies the required universal property in the category of 
$\mathcal V$-categories, taking $*$ to be the $\mathcal V$-category with one object, with homset $I$.
Denote by $\iota_{\EuScript C}:\EuScript C\rightarrow\EuScript E$ the natural morphism and similarly for $\iota_{\EuScript D}$.
Given a $\mathcal V$-category $\EuScript F$, $F:\EuScript C\rightarrow\EuScript F$ and two $\mathcal V$-functors $G:\EuScript D\rightarrow\EuScript F$ 
 satisfying $F(\omega_P)=G(\alpha_Q),$ we prove that there exists a unique $\mathcal V$-functor $H:
\EuScript E\rightarrow \EuScript F$ such that $H\iota_{\EuScript C}=F,\ H\iota_{\EuScript D}=G$. The functor $H$ is clearly uniquely defined on objects, and on 
most of the morphisms. For $ x\in P\setminus\{\omega_P\} \text{ and } y\in Q\setminus\{\alpha_Q\}$, we (have to) define $H$ as the composite
\[\xymatrix{
\EuScript C(x,\omega_P)\otimes \EuScript D(\alpha_Q,y)\ar[r]^-{F\otimes G}&\EuScript F(F(x),F(\omega_P))\otimes
\EuScript F(G(\alpha_Q),G(y))\ar[r]^-{\circ} &\EuScript F(F(x),G(y)),}\]
and we check easily that this defines an enriched functor.
\end{proof}

\begin{cor}\label{C:nerveconcat} Let $P$ and $Q$ be two bounded posets. If $\EuScript C$ is a $P$-shaped poset-category and $\EuScript D$ is a $Q$-shaped poset-category,
then $N(\EuScript C)\vee N(\EuScript D)$ is isomorphic to $N(\EuScript C\vee\EuScript D)$ as simplicial categories.
\end{cor}

\begin{proof} It follows directly from the explicit description of $\EuScript C\vee\EuScript D$ in Proposition \ref{P:V-concat} and from the fact that 
$N(A\times B)\cong N(A)\times N(B)$ for any posets $A$ and $B$.
\end{proof}

\begin{prop}\label{P:sigmacolim} Let $\EuScript C$ be a $P$-shaped simplicial category, with $P$ a bounded poset. Let 
$\varphi:\EuScript C(\alpha_P,\omega_P)\rightarrow Y$ be a morphism of simplicial sets.
Denote by $\EuScript D$ the colimit of the pushout diagram
\[\xymatrix{\Sigma Y&\Sigma \EuScript C(\alpha_P,\omega_P)\ar[l]_-{\Sigma\varphi}\ar[r]&\EuScript C},\]
where the right arrow is the counit of the adjunction of Proposition \ref{adj_Sigma_Hom}.
The simplicial category $\EuScript D$ is  $P$-shaped and has  simplicial sets of morphisms $\EuScript C(a,b)$, if 
$(a,b)\not=(\alpha_P,\omega_P)$ and $Y$ if $(a,b)=(\alpha_P,\omega_P)$. \\
The composition $\EuScript D(b,c)\otimes \EuScript D(a,b)\to \EuScript D(a,c)$ is that of $\EuScript C$ if $(a,b)\not=(\alpha_P,\omega_P)$ and is the composition in 
$\EuScript C$ followed by $\varphi$ if $(a,b)=(\alpha_P,\omega_P)$.
\end{prop}
\begin{proof}
Note that since the functor $Ob:s\Cat\rightarrow\Set$ is cocontinuous, the set of objects of the colimit of the diagram is in bijection with $P$. Moreover the description of the hom 
sets of $\EuScript D$ shows that the category $\EuScript D$ is $P$-shaped.
We check that the proposed simplicial category $\EuScript D$ satisfies the universal property of pushouts. 
Let $\EuScript E$ be a simplicial category with $F:\EuScript C\rightarrow \EuScript E$ and $g:Y\to\EuScript E(F(\alpha_P),F(\omega_P))$ 
such that $g\circ\varphi=F_{\alpha_P,\omega_P}$. There is a unique way to build $H:\EuScript D\rightarrow \EuScript E$ making the 
diagrams commute. It coincides with $F$ on objects as well as on morphisms $\EuScript D(a,b)\rightarrow \EuScript E(F(a),F(b))$ for 
$(a,b)\not=(\alpha_P,\omega_P)$. In addition, we have $H_{\alpha_P,\omega_P}=g$. It is easy to check that $H$ is a morphism of 
simplicial categories, using the equation relating $g$ and $F$.
\end{proof}

\begin{rmk}
Proposition \ref{P:sigmacolim} holds more widely for all $P$-shaped $\mathcal V$-categories.  All it uses is the adjunction $\Sigma \vdash \Hom$ of 
Proposition \ref{adj_Sigma_Hom}. But the latter (as a mere adjunction) holds in the general setting of bipointed $P$-shaped 
$\mathcal V$-categories. Moreover, when restricted to $[1]$-shaped $\mathcal V$-categories, the adjunction is an equivalence. \end{rmk}

\section{Combinatorics} \label{combinatorics-appendix}

\subsection{The category  $\SubNeck(T)$, with $T$ a necklace}\label{S:B1}

\begin{prop}\label{P:subneckposet} The category $\SubNeck(\square^n)$ is a poset. It is in bijection with the poset of ordered partitions of $\{1,\ldots,n\}$, 
where the order $A\leq_r B$ is the refinement inverse order, defined as the reflexive and transitive closure of
\[ (A_1;\ldots;A_k)\leq_r (A_1;\ldots;A_i\cup A_{i+1};\ldots;A_k).\]
In particular, it has a greatest element, the partition with 1 block $(\{1,\ldots,n\})$, and the set of minimal elements is naturally in bijection with the 
symmetric group $\Sigma_n$. This poset admits all least upper bounds.
\end{prop}

\begin{proof} (See also \cite[Section 7]{Zie17} and  \cite[Section 7]{Zie20},  where elements of $\SubNeck(\square^n)$ are identified with 
cube chains). As seen in Lemma \ref{L:canonical-embedding}, an object $\varphi:\square^{n_1}\vee\ldots\square^{n_k}
\hookrightarrow \square^n$ is determined 
\begin{itemize}
\item by a sequence $s_\varphi=(\alpha=a_0\prec a_1\prec\ldots\prec a_k=\omega),$ 
with $a_i\in(\square^n)_0$, $d(a_{i-1},a_i)=n_i$ and $n_1+\ldots+n_k=n$, or, equivalently,
\item
by an ordered partition of $\{1,\ldots,n\}$: setting $A_i=a_i\setminus a_{i-1}$, we get  $A_\varphi:=(A_1;\ldots;A_k)$.
\end{itemize}
We denote the set $\{a_0,\ldots,a_k\}$ arising from  a sequence  $s$ as above by $\overline{s}$. 
The following easy verifications are left to the reader.
\begin{itemize}
\item
Given $\varphi,\psi$  in $\SubNeck(\square^n)$, there is a morphism from $\varphi$ to $\psi$ if and only 
$\overline{s_\psi} \subseteq \overline{s_\varphi}$, and this morphism is unique.
\item
We have  $\overline{s_\psi} \subseteq \overline{s_\varphi}$ if and only if $A_\varphi\leq_r A_\psi$.
\end{itemize}
In particular, if $A$ is a set of sequences $\{s_1,\ldots,s_l\}$ then its least upper bound is the sequence $s$ associated to 
$\cap_{i=1}^l \overline{s_i}$.
\end{proof}

The following corollary is a direct consequence of the previous proposition and of Proposition \ref{P:Nec_morphisms}.

\begin{cor}\label{C:subneckposet} Let $T=\square^{n_1}\vee\dots\vee\square^{n_k}$ be a necklace. The category $\SubNeck(T)$ is a poset, 
whose poset relation is denoted  $\leq_r$. It is the product of the categories $\SubNeck(\square^{n_i})$. In particular, it has a greatest element and admits all least upper bounds.
\end{cor}

\begin{defn}\label{D:closed} Let $(P,\leq)$ be a poset. A subset $A$ of $P$ is {\sl downward closed} if for all $y\leq x$ in $P$ we have $x\in A\Rightarrow y\in A$.
It is {\sl upward closed} if it is downward closed in the opposite poset of $P$. 
\end{defn}

\begin{lem}\label{L:classical}
Let $(P,\leq)$ be a poset and $A$ and $B$ upward closed subsets of $P$.
Then $N(A\cup B) \cong \colim(N(A)\hookleftarrow N(A\cap B)\hookrightarrow N(B))$.
\end{lem}
\begin{proof} The hypotheses imply that $N(A\cup B)$ satisfies the universal property of the pushout diagram.
\end{proof}

\begin{lem}\label{lem_subneck}
Let $T$ be a necklace. If $\EuScript{A} \subseteq \SubNeck(T)$ is downward closed (for the order $\leq_r$), then
the canonical morphism $\colim\limits_{U\in\EuScript{A}}\Cbox_t(U) \to \Cbox_t(T)$ is a monomorphism of simplicial sets.
\end{lem}

\begin{proof}
We only need to examine the situation of two $n$-simplices $u,v$ coming from different subnecks $U,V$ in $\EuScript{A}$ and whose images in 
$\Cbox_t(T)$ are identified:
\[\begin{tikzcd}
\Delta^n \ar[r, "u"] \ar[d, "v"] & \Cbox_t(U) \ar[d] \\
\Cbox_t(V) \ar[r] & \Cbox_t(T)
\end{tikzcd}\]
The map $\varphi: \Delta^n\to\Cbox_t(T)$ gives a set $A_\varphi$ of $n+1$ paths of $T$ (by Theorem \ref{rigid_neck}) with values both in $U$ and $V$. 
Let $W$ be the upper bound of $A_\varphi$ in $\SubNeck(T)$, provided by the preceding corollary. Since $U$ and $V$ are upper bounds of $A_\varphi$, 
we have $W\leq_r U$ and $W\leq_r V$ in $\SubNeck(T)$.
Hence there is a factorisation in the diagram:
\[\begin{tikzcd}[row sep=-1pt]
\Delta^n \ar[rr, "u"] \ar[dd, "v"] \ar[dr, dashed, "w"] &&     \Cbox_t(U) \ar[dd] \\
&\Cbox_t(W) \ar[dl, dashed] \ar[ur, dashed] \ar[dr] \\
\Cbox_t(V) \ar[rr] && \Cbox_t(T)
\end{tikzcd}\]
Moreover we have $W \in \EuScript{A}$, since $\EuScript{A}$ is downward closed.  Thus the diagram says that $u,v,w$ are identified in the colimit, which completes the proof.
\end{proof}

\subsection{On the homotopy type of  $\SubNeck(\sqcap^n_{0,n}) $}

To simplify the notation in this section, for $n\geq 1$,  we will denote by $P_n$ the poset of ordered partitions of $\{1,\dots,n+1\}$, that is, $P_n=\SubNeck(\square^{n+1})$. 
Similary, we set $\partial P_n= \SubNeck(\partial\square^{n+1})$ and $\sqcap P_n=\SubNeck(\sqcap^{n+1}_{0,n+1})$ (see Definition \ref{D:subneck}).
Note that
\[ \partial P_n = P_n\setminus{(\{1,\dots,n+1\})}\quad \text{ and } \quad \sqcap P_n = \partial P_n\setminus{(\{1,\dots,n\};\{n+1\})}.\]

The nerve of $P_n$ is contractible, since $P_n$ has a greatest element.

The next proposition is what we need for Proposition \ref{C_hn}.

\begin{prop}\label{C:nerf_subneck} For every $n\geq 2$, the nerve of $\SubNeck((\sqcap^n_{0,n})$ is contractible.
\end{prop}

\begin{proof} In our notation, we have to prove that, for $n\geq 1$, the nerve of $\sqcap P_n$ is contractible.
For $n=1$, the poset $\sqcap P_1$ is a singleton, namely the ordered partition $(\{2\};\{1\})$, hence is contractible.
Assume $n\geq 2$. Let us fix some notation:
\begin{itemize}
\item For an ordered partition $x = (A_1;\dots;A_k)\in P_n$, and $0\leq l\leq k$, we set 
\[m_l(x):=\#(A_1\cup\dots\cup A_{l}),\] 
with the convention that $m_0(x)=0$. Note that $m_k(x)=n+1$.  
   \item For an ordered partition $x = (A_1;\dots;A_k)\in P_n$ with $n+1\in A_l$, we set \[\alpha(x) := m_{l-1}(x) \text{ and } \beta(x) := m_l(x).\]
    \item We have $\partial P_n = T_0 \cup T_1 \cup \dots \cup T_n$ with \[T_i := \{x \in \partial P_n | \alpha(x) \leq i < \beta(x)\}.\]
    \item For every $0\leq i\leq n-1$ we set
     \[T_{i,i+1} := T_i \cap T_{i+1} = (T_0\cup \dots \cup T_i)\cap T_{i+1}.\]

\end{itemize}
We note that $\sqcap P_n = T_0 \cup T_1 \cup \dots T_{n-1}\cup T'_n$ with $T'_n=T_n\cap (\sqcap P_n)$. 
We also note (by a simple case analysis) that $T_0,\ldots,T_{n-1}$ and $T'_n$ are upward closed, allowing us to use   Lemma \ref{L:classical} repeatedly and get
\[N(\sqcap P_n) \cong \colim\left(\begin{tikzcd}[column sep=-4pt]
N(T_0) & & N(T_1) &\ldots&\ldots&\ldots & N(T_{n-1}) & & N(T_n')
\\ & N(T_{0,1}) \ar[ul, hook]\ar[ur, hook] & & \ar[ul, hook]& \dots  &\ar[ur, hook] && N(T_{n-1,n})\ar[ul, hook]\ar[ur, hook] & 
\end{tikzcd}\right)\]

Using  \cite[Proposition 19.9.1]{H}, we conclude that the colimit of this diagram is Kan-Quillen equivalent to its homotopy colimit. Indeed, let $D$ be 
the underlying category of the diagram, which has objects $x_i$ for $0\leq i\leq n$ and $y_{i,i+1}$ for $0\leq i<n$ and morphisms from $y_{i,i+1}$ to 
$x_i$ and $x_{i+1}$. We endow $D$ with the following Reedy structure:  $x_i$ has degree $2i$ and $y_{i,i+1}$ has degree $2i-1$. Then $D$ has fibrant 
constants and the diagram above is Reedy cofibrant.
In consequence, we focus our attention on  the homotopy type of 
$N(T_i)$ for $0\leq i<n$, $N(T'_n)$ and $N(T_{i,i+1})$ for $0\leq i<n$.

Let $\pi : P_n \to P_{n-1}$ be the poset morphism removing $(n+1)$ from the ordered partition. Note that $\pi(T_i)\subset \partial P_n$ for every $1\leq i\leq n-1$ since
\[\pi^{-1}(\{1,\ldots,n\})=\{ (\{1,\ldots,n+1\}), (\{n+1\};\{1,\ldots,n\}), (\{1,\ldots,n\};\{n+1\})\},\]
and since none of the elements in this set lies in $T_i$, for $1\leq i\leq n-1$.
Note that given a partition $x=(A_1;\ldots;A_k)$ of $P_{n-1}$ and $0\leq i\leq n-1$, there exists a unique $l\in\{0,\ldots,k-1\}$ such that $m_{l}(x)\leq i<m_{l+1}(x)$.
We leave it to the reader to check the following facts.
\begin{itemize}
 \item For $0\leq i < n$, the induced map $T_{i,i+1}\stackrel{\pi}\to \partial P_{n-1}$ is an isomorphim of posets with inverse
\[x=(A_1;\dots;A_k)\in \partial P_{n-1} \mapsto 
    				(A_1;\dots;A_{l+1}\cup\{n+1\};\dots;A_k),\]
where  $m_l(x) \leq i <m_{l+1}(x)$.

\item For $0<i<n$, the map $T_i \stackrel{\pi}\to  \partial P_{n-1}$ is an adjunction of posets with a section $\sigma$ given for $x=(A_1;\ldots;A_k)\in \partial P_{n-1}$ by 
\[
\sigma(x)=\begin{cases}
    				(A_1;\dots,A_l;\{n+1\};A_{l+1};\dots;A_k), & \text{ if } m_l(x)=i, \\
    				(A_1;\dots;A_{l+1}\cup\{n+1\};\dots;A_k), & \text{ if } m_l(x)<i<m_{l+1}(x).
		\end{cases}
\]
(Indeed, $\pi\circ\sigma = \id$ and $\sigma\circ\pi \leq \id$.)

\item   Similarly, the maps  $T_0 \stackrel{\pi}\to P_{n-1}$ and $T'_n \stackrel{\pi}\to \partial P_{n-1}$ are adjunctions of posets with the respective sections 
$$\begin{array}{l}
(A_1;\dots;A_k)\in P_{n-1} \mapsto (\{n+1\};A_1;\dots;A_k),\\
(A_1;\dots;A_k)\in \partial P_{n-1} \mapsto (A_1;\dots;A_k;\{n+1\}).
\end{array}$$
\end{itemize}

Putting everything together, and using the fact that an adjunction of posets gives rise to a homotopy equivalence, and hence to a Kan-Quillen equivalence 
between their nerves, we have:
\[\begin{array}{lllllll}
N(\sqcap P_n)& \sim & \hocolim\left(\begin{tikzcd}[column sep={3.5em,between origins}]
N(P_{n-1}) & & N(\partial P_{n-1}) &\ldots&  N(\partial P_{n-1}) & & N(\partial P_{n-1})
\\ & N(\partial P_{n-1}) \ar[ul, hook]\ar[ur, equal] & & \ldots  && N(\partial P_{n-1})\ar[ul, equal]\ar[ur, equal] & 
\end{tikzcd}\right)\\\\
& \sim & \hocolim \left( \begin{tikzcd}[column sep=0pt]
N(P_{n-1})  
\\  N(\partial P_{n-1}) \ar[u, hook]
\end{tikzcd} \right) \sim
 \colim \left( \begin{tikzcd}[column sep=0pt]
N(P_{n-1}) &
\\  N(\partial P_{n-1}) \ar[u, hook]
\end{tikzcd} \right)  \sim N(P_{n-1}).
\end{array}\]
Hence the nerve of  $\sqcap P_n$ is contractible.
\end{proof}
\begin{rmk} Related results relative to $P_n$ and $\partial P_n$ can be found in, say, \cite{Zie17} or in \cite{Zieg95} (where they are put to use to establish 
connections between cubical sets and higher dimensional automata, through directed path spaces): the geometric realisations of the ordered partition 
posets are the permutohedra, which are homeomorphic to a ball. Moreover, the  pair $(|N(P_n)|,|N(\partial P_n)|)$ is homeomorphic to the pair $(D^n, S^{n-1})$.
\end{rmk}

\bibliographystyle{abbrv}

\bibliography{cubical}
\end{document}